\documentclass[reqno,a4paper,12pt]{amsart}
\usepackage{amscd,amssymb,amsmath,amsfonts,mathrsfs}
\usepackage{hyperref}

\usepackage{color}
\definecolor{darkgreen}{cmyk}{1,0,1,.2}
\definecolor{m}{rgb}{1,0.1,1}

\newcommand\supp{\operatorname{supp}}

\newcommand\rank{\operatorname{rank}}

\newcommand\Kp{{\mathscr{K}}}

\renewcommand\lq{\leqslant}
\newcommand\gq{\geqslant}
\newcommand\E{\mathcal{E}}

\newcommand\diag{\operatorname{diag}}

\newcommand\ds{\displaystyle}

\renewcommand\Im{\operatorname{Im}}
\newcommand\C{\mathbb C}

\renewcommand\r{{\mathcal{R}}}

\newcommand\De{\Delta}

\renewcommand\H{\mathcal{H}}

\newcommand\N{\mathbb N}

\newcommand\R{\mathbb R}
\newcommand\st{\text{ such that }}
\newcommand\Z{\mathbb Z}

\newcommand\MM{\mathcal{M}}
\renewcommand\L{\mathcal{L}}
\newcommand\D{\mathcal D}

\newcommand\G{\mathcal{G}}
\newcommand\ts{{\otimes}}
\newcommand\rt{{\rtimes}}
\newcommand\lt{{\ltimes}}
\newcommand\rtr{{\rtimes_{r}}}

\newcommand\la{\lambda}

\newcommand\Ga{{\Gamma}}

\newcommand\ga{{\gamma}}
\newcommand\al{{\alpha}}
\newcommand\lto{{\longrightarrow}}
\newcommand\defi{{\stackrel{\text{def}}{=\!=}}}

\newcommand\M{{M}}
\newcommand\U{\operatorname{U}}
\renewcommand\P{\operatorname{P}}
\newcommand\ka{\kappa}
\newcommand\ue{\operatorname{U}_n^{\varepsilon,E}}
\newcommand\pe{\operatorname{P}_n^{\varepsilon,E}}
\newcommand\eps{\varepsilon}

\theoremstyle{plain}
\newtheorem{theorem}{Theorem}[section]
\newtheorem{proposition}[theorem]{Proposition}
\newtheorem{corollary}[theorem]{Corollary}

\newtheorem{lemma}[theorem]{Lemma}
\newtheorem{definition}[theorem]{Definition}
\theoremstyle{definition}
\newtheorem{remark}[theorem]{Remark}
\newtheorem{example}[theorem]{Example}
\newtheorem{notation}[theorem]{Notation}
\begin{document}
\title[Groupoids decomposition,  propagation and $K$-theory]{Groupoids decomposition,  propagation and operator $K$-theory }
 \author[H. Oyono-Oyono]{Herv\'e Oyono-Oyono}
 \address{Universit\'e de Lorraine, Metz , France}
\email{herve.oyono-oyono@math.cnrs.fr}

\maketitle

\begin{center}
{\it Dedicated to the memory of Etienne Blanchard}
\end{center}

\begin{abstract}
In this paper, we streamline the technique of groupoids coarse decomposition for purpose of $K$-theory computations of groupoids crossed products. This technique was first
introduced by Guoliang Yu in his proof of   Novikov conjecture for groups with finite
asymptotic dimension. The main tool we use for these  computations  is  controlled operator $K$-theory.
\end{abstract}

\begin{flushleft}{\it Keywords: Groupoids,
   Operator   $K$-theory,  Coarse Geometry, Baum-Connes Conjecture.}

\medskip

\end{flushleft}

\tableofcontents
\section{Introduction}
The concept of coarse decomposability for locally compact groupoids  was introduced by several authors
\cite{bd,gwy2,oy4,w} in order to compute  $K$-theory of  reduced $C^*$-algebras and of reduced crossed product algebras of locally compact groupoids. It generalizes the "cut-and-pasting" technique developed by G. Yu in \cite{yuasymp} to prove the Novikov conjecture for groups with finite asymptotic dimension. The "cut-and-pasting"  has been then extended by E.Guentner, R.Tessera and G. Yu in  \cite{gty1} in order to  study topological rigidity of manifolds and in \cite{gty2} in their approach of the Borel conjecture. In these works, they consider a class of finitely
 generated groups which satisfy a metric property called { finite decomposition complexity}.   This property can be interpreted in term of decomposition complexity  of the coarse groupoid associated to the metric space, which leads naturally to extend this notion to locally compact groupoids. A first generalization  was provided by E. Guentner, R. Willet and G. Yu in \cite{gwy1} in order to study the dynamical properties of finitely generated  group actions  on locally compact spaces with   at each  order  a   one step  decomposition  into pieces with "finite dynamic".  In \cite{gwy2},  the same authors consider the case of finitely generated   group actions 
 on locally compact space  which,  given a sequence of  orders, decompose in a finite number of steps   into pieces with "finite dynamic". They give  a new proof of the Baum-Connes conjecture (with trivial coefficients) for these action groupoids. This approach rises a large amount of interest since it does not involve infinite dimension analysis and can be generalized to computations in non $C^*$-algebraic situations (for instance to $\ell^p$-crossed products as considered in \cite{ch}).  The main tool used in this proof is quantitative $K$-theory.
Quantitative  $K$-theory was first introduced in \cite{yuasymp} for obstruction algebras in order to prove the Novikov conjecture for groups with finite asymptotic dimension. It has been then extended in  \cite{oy2} to the setting of $C^*$-algebras equipped with a filtration arising from a length and  in \cite{dell} to  the general framework of $C^*$-algebras filtered by  an abstract coarse structure. In \cite{oy4} was stated a controlled Mayer-Vietoris exact sequence in quantitative $K$-theory  associated to decomposition in "ideals at order $r$"  which turned out to be tailored for $K$-theory computations under groupoid  in decomposability (see \cite{dell} for the extension  to general filtrations). It has been apply in \cite{bd}, to the K\"unneth formula in $K$-theory for groupoid $C^*$-algebras and crossed  product algebras.  Loosely speaking, we consider  decomposition of the  set of elements  of a given order of a groupoid as the union of two open subgroupoids  (see Definition \ref{definition-remark}). Following \cite{gwy2}, we say that a locally compact groupoid $\G$ has finite complexity decomposition with respect to a family $\D$ of open subgroupoids if starting with a given sequence  of orders, then iterating the above decomposition ends up with elements belonging to  $\D$ in a finite number of steps (see  Definition \ref{definition-finite-decomposition}). The main result of this paper is the following:

\smallbreak

Let $\G$ be locally compact groupoid with  finite decomposition complexity with respect to a family $\D$ of relatively clopen subgroupoids (see Definition \ref{definition-relatively-clopen}) and let $f:A\to B$ a homomorphism of $\G$-algebras. If the morphism
$$K_*(A\rtr\H)\to K_*(B\rtr \H)$$ induced in $K$-theory by $f$ is a an isomorphism for any $\H$ in $\D$, then so is $$K_*(A\rtr \G)\to K_*(B\rtr \G).$$

\smallbreak

 We then extend this result to  morphisms induced  by elements in
$KK_*^\G(A,B)$  and we give an application of this result to the Baum-Connes conjecture for locally groupoids admitting a $\ga$-element in sense of  \cite{tuhyp}.
These stability results  should be compare with those obtained by R. Willet in \cite{w} using different technics.

\smallbreak

{\it Outline of the  paper.}  In Section $1$, we first recall some basic definitions
concerning locally compact groupoids and their actions. We introduce  the notion 
of $\G$-order for a locally compact groupoid $\G$ which can be viewed as the generalization  both of a length on a group and of a distance on a proper metric space.
Following the idea of \cite[Definition 3.14]{gwy2}, we then introduce the notion of
$\r$-decomposition for a $\G$-order $\r$, which leads to the concept of $\D$-decomposability of an open subgroupoid of $\G$ with respect to a family 
$\D$ of open subgroupoids and to $\D$-{\it finite decomposition complexity} ($\D$-fdc), generalizing  { finite dynamical complexity} defined in  \cite{gwy2}.

Section $2$ is devoted to some reminders on groupoid actions on $C^*$-algebras and their reduced crossed product algebras. 

In Section $3$, we introduce the primary  tool for the proof of  our main theorem, the controlled Mayer-Vietoris exact sequence in quantitative $K$-theory associated to a groupoid decomposition. We first review
from \cite{dell} the main features of quantitative $K$-theory for $C^*$-algebra filtered by an 
abstract coarse structure. We observe that $\G$-orders provide such a structure on 
crossed products algebras of a groupoid $\G$. We recall the definition of a controlled Mayer-Vietoris pair and we show that groupoid decompositions of order $\r$  give rise to 
controlled Mayer-Vietoris pairs. Eventually, we recall the statement of the 
controlled Mayer-Vietoris exact sequence in quantitative $K$-theory associated to a
controlled Mayer-Vietoris pair.

We prove in Section $4$ the main result of the paper. Although the proof is tedious, the principle is quite simple as it is the extension of the five lemma   to the setting of controlled exact sequences. We then extend our main result to the case of the morphisms induced
in $K$-theory by elements of $\G$-equivariant $KK$-theory. This is done by noticing that  every such elements is up to $KK$-equivalence given by a equivariant homomorphism.

In Section $5$, we give some applications to the Baum-Connes conjecture for locally compact groupoids.
We first recall from \cite{tuhyp} the statement of the Baum-Connes conjecture in the setting of locally compact groupoids and the definition of  $\ga$-elements. We end this section with heridity results of the Baum-Connes conjecture for groupoids with 
$\D$-{finite decomposition complexity} which admit a $\ga$-element in sense of 
\cite{tuhyp}.

\bigbreak

I  would like to thank warmly  J. Renault for the very helpful  discussions  we had concerning  relatively clopen subgroupoids. I am grateful to him for the comments and suggestions he made after carefully  reading this paper. For the occasion of his recent retirement, I  would like to express my deep admiration for him.

\bigbreak

This paper is dedicated to the memory of Etienne Blanchard from whom I learned almost everything I know about $C(X)$-algebras.
\section{Coarse decomposition for groupoid}
Coarse decomposability  for locally groupoids is the generalization of the concept  of 
decomposability for a family of metric spaces introduced in \cite{gty2}.  
In this section, after some reminders concerning locally groupoids and theirs actions, we introduced for a locally compact groupoid $\G$ the notion of $\G$-orders 
generalizing on one hand distances on metric spaces and on the other hand  lengths on groups.  Following ideas of \cite{gwy2}, this allows to define  decomposition of order $\r$ for a subgroupoid of $\G$ which leads naturally to coarse decomposability and to finite  decomposition complexity with respect to a set of open subgroupoids of $\G$.
\subsection{Groupoids}
We assume that the reader is familiar with the basic definition concerning groupoids. For more details, we refer to \cite{pat,ren}.

\smallbreak

A groupoid with space of units $X$ consists of a set $\G$  provided with
\begin{itemize}
\item two maps $s:\G\to X$ and $r:\G\to X$  respectively called the source map and the range map;
\item a map $u:X\to\G;\; x\mapsto u_x$ which is a section both for $s$ and $r$;
\item an associative  composition $$\G\times_X\G\to \G:\;(\ga,\ga')\mapsto \ga\cdot\ga',$$ with 
$$\G\times_X\G=\{(\ga,\ga')\in\G\times\G\text{ such that }s(\ga)=r(\ga')\}$$  such that
$$s(\ga\cdot\ga')=s(\ga')$$ and $$r(\ga\cdot\ga')=r(\ga)$$  for any $(\ga,\ga')$ in 
$\G\times_X\G$  and $$\ga\cdot u_{s(\ga)}=u_{r(\ga)}\cdot \ga=\ga$$ for any $\ga$ in $\G$;
\item an inverse map $$\G\to\G;\, \ga\mapsto \ga^{-1}$$ such that $$s(\ga^{-1})=r(\ga),$$$$r(\ga^{-1})=s(\ga),$$ $$\ga\cdot \ga^{-1}=u_{r(\ga)},$$  and $$\ga^{-1}\cdot \ga=u_{s(\ga)}$$ for any $\ga$ in $\G$.
\end{itemize}
\begin{notation}
Let $\G$ be a groupoid with  space of units   $X$ and source and range maps $s,r:\G\to X$.
\begin{itemize}
\item Let $Z$ be a subset of $\G$.
\begin{itemize}
\item we set $Z^{-1}=\{\ga^{-1};\, \ga\in Z\}$;
\item for any $Y\subseteq X$, we set $Z_Y=s^{-1}(Y)\cap Z$ and $Z^Y=r^{-1}(Y)\cap Z$;
\item for any subsets $Y_1$ and $Y_2$ of $X$, we set $Z^{Y_2}_{Y_1}=Z^{Y_2}\cap Z_{Y_1}$;
\end{itemize}
\item Let $Z_1$ and $Z_2$ be subsets in $\G$, we set $$Z_1\cdot Z_2=\{\ga_1\ga_2;\,\ga_1\in Z_1,\,\ga_2\in Z_2\text{ and }s(\ga_1)=r(\ga_2)\}.$$
\end{itemize}
\end{notation}

A locally compact groupoid is a  groupoid provided with locally compact topology and such that the structure maps a continuous. In this paper, all the groupoids are assume to be locally compact and Hausdorff.  An open subgroupoid of $\G$ is a subgroupoid $\H$ of $\G$ which is open as a subset and such that the space of units is open in the space of unit
of $\G$. Notice that the latter condition always holds if the source map of $\G$ is open,
for instance if $\G$ is provided with a Haar  system \cite[Lemma 6.5]{tuhyp}.

\begin{definition}\label{definition-relatively-clopen}
Let $\G$ be locally compact groupoid. A relatively clopen subgroupoid of $\G$  is an open subgroupoid   $\H$ of $\G$ such that  if $Y$ stands for the unit space of $\H$, then $\H$ is closed in $\G_Y$.
\end{definition}
\begin{remark}\label{remark-relatively-clopen} Let $\G$ be locally compact groupoid and let  $\H$  be   a relatively clopen subgroupoid of $\G$  with unit space $Y$. Then
$\H$ is clopen in $\G^Y$ and in $\G_Y^Y$.
\end{remark}
Next lemma is  straightforward to prove.
\begin{lemma}\label{lemma-locally-clopen}Let $\G$ be  a locally compact groupoid  and let $\H$ be  an open subgroupoid  of $\G$ with  unit space $Y$. Then $\H$ is relatively  clopen if and only if   $K\cap\H$ is compact for any compact subset $K$ of $\G_Y$.
\end{lemma}
We recall that a locally compact groupoid with space of units $X$ is {\bf proper} if the map
$$\G\to X\times X;\, \ga \mapsto (r(\ga),s(\ga))$$ is proper.
As a consequence of Lemma \ref{lemma-locally-clopen}, we obtain te following corollary.
\begin{corollary}\label{corollary-subgroupoid-proper}
 Let $\G$ be  a locally compact proper groupoid. Then  relatively clopen subgroupoids of $\G$ are proper.
 \end{corollary}
\subsection{Groupoid actions}
Let us recall first the definition of a (left) action of a groupoid.
Let $\G$ be a groupoid with space of units $X$ and source and range maps $s$ and $r$.
An action of the groupoid $\G$ on a set $Z$ consist of a map
$p:Z\to X$ called the {\bf anchor map} and a map
$$\G\times_X Z\to Z;\, (\ga,z)\mapsto \ga\cdot z,$$ with 
$$\G\times_X Z=\{(\ga,z)\in\G\times Z\text{ such that }s(\ga)=p(z)\}$$ such that
\begin{enumerate}
\item for any $\ga$ and $\ga'$ in $\G$ and $z$ in $Z$ such that  $(\ga,\ga')$ is  in $\G\times_X\G$ and $(\ga',z)$ is in $\G\times_X Z$, then
$(\ga,\ga'\cdot z)$ belongs to $\G\times_X Z$ and $\ga\cdot (\ga'\cdot z)=(\ga\cdot \ga')\cdot z$;
\item $u_{p(z)}\cdot z=z$ for any $z$ in $Z$.
\end{enumerate}
Notice that these conditions implies that
$p(\ga\cdot z)=r(\ga)$  and $\ga^{-1}\cdot\ga\cdot z=z$ for any $(\ga,z)$ in
$\G\times_X Z$.  If $x$ is an element in $X$, then the fiber of $Z$ at $x$ is $Z_x\defi f^{-1}(\{x\})$. If $\G$ is a locally compact groupoid and if $Z$ is a locally compact space,
we require the anchor map and the action map to be continuous. In this case, $Z$ will be called a $\G$-space. In what follows, all   $\G$-spaces are supposed to be Hausdorff.
If $Z$ and $Z'$ are $\G$-space with anchor maps $p_Z$ and $p_{Z'}$, a map $f:Z\to Z'$ is called a $\G$-map if $f$ is continuous, $p_{Z'}\circ f=p_Z$  and $f(\ga\cdot z)=\ga\cdot f(z)$ for all $(\ga,z)$ in $\G\times_X Z$. 

Let $\G$ be a groupoid  with space of units $X$ acting on a set $Z$ with anchor map
$p:Z\to X$. Then the action groupoid corresponding to the action of $\G$ on $Z$  denoted by $\G\lt Z$ is the set $\G\times_X Z$, with $Z$ as space of units with source map
$$\G\lt Z\to Z;\, (\ga,z)\mapsto z$$ and range map
$$\G\lt Z\to Z;\, (\ga,z)\mapsto \ga\cdot z,$$ unit map
$$Z\to\G\lt Z;\, z\mapsto (u_{p(z)},z),$$ composition
$$(\G\lt Z)\times_Z(\G\lt Z);\, (\ga,\ga' z)\cdot (\ga',z)\mapsto (\ga\ga',z)$$
and inverse
$$\G\lt Z\to\G\lt Z;\, (\ga,z)\mapsto(\ga^{-1},\ga\cdot z).$$
If $\G$ is a locally compact groupoid and $Z$ is a $\G$-space, then
$\G\lt Z$ is a locally compact groupoid.
A $\G$-space $Z$ is called proper (or the action of $\G$  on $Z$ is said to be proper) if the action groupoid $\G\lt Z$ is proper.
\begin{remark}\label{remark-cross-product-action}
Let $\G$ be a locally compact groupoid with  space of units  $X$ acting on a locally compact space $Y$. 
% with respect  to an anchor map $q_Z:Y\to X $.
\begin{enumerate}
\item a $G\lt Y$-space is precisely a $\G$-space $Z$ together with a $\G$-map  $f:Z\to Y$;
\item in this case, $$\G\lt Z\lto (\G\lt Y)\lt Z;(\ga,z)\mapsto (\ga,f(z),z)$$ is a groupoid isomorphism;
\item in consequence, a $\G\lt Y$-space $Z$ is  proper if and only if it  is  
proper as a $\G$-space;
\item in particular, if $Z$ is a proper $\G$-space, then $Z\times_X Y$ is a proper 
$\G\lt Y$-space  with anchor map 
     given by the projection on the second factor (here $Z\times_X Y$ stands for the fiber product over the two anchor maps).

\end{enumerate}
\end{remark}

\begin{remark}\label{remark-relatively-clopen-action} Let $\G$ be locally compact groupoid and let  $\H$  be   a relatively clopen subgroupoid of $\G$  with unit space $Y$.
For any left $\G$-space $Z$, then $\H\rt Z_Y$ is relatively clopen in $\G\rt Z$.
\end{remark}
\subsection{Induced actions}\label{subsec-induced-actions}

We recall now from \cite{bonicke} the notion of induced action to a groupoid  from a subgroupoid action.
Let $\G$ be a locally compact groupoid with  space of units  $X$ and open source and range maps, let $\H$ be a relatively
clopen subgroupoid of $\G$ with space of units $Y$ and let $Z$ be a (left) $\H$-space with anchor map $p:Z\to Y$. Let us define
on $$\G\times_YZ\defi \{(\gamma,z)\in \G\times Z\text{ such that }  s(\ga)=p(z)\}$$ the $\H$-action with anchor map
$$\G\times_YZ\to Y:(\gamma,z)\mapsto p(z)$$ by $$\ga\cdot (\ga',z)=(\ga'\ga^{-1},\ga z)$$ for any $\ga$ in $\H$ and 
$(\ga',z)$ in $\G\times_YZ$ such that $s(\gamma)=p(z)$. The $\H$-action defined in this way is proper and the quotient space $$\G\times_\H Z\defi (\G\times_YZ)/\H$$ is Hausdorff and  locally compact. Let us denote by $[\ga,z]$ the class in $\G\times_\H Z$ of  an element $(\ga,z)$ in 
$\G\times_\H Z$. Then  $\G\times_\H Z$ is provided with a $\G$-action with anchor map $$\G\times_\H Z\to X:[\gamma,z]\mapsto r(\ga)$$ defined by
$$\ga\cdot[\ga',z]=[\ga\ga',z]$$ for any $\ga$ in $\G$ and $[\ga',z]$ in $\G\times_\H Z$ such that $s(\ga)=r(\ga')$ and is  called the $\G$-space induced by the $\H$-space $Z$. 
\begin{proposition}\cite{bonicke}
 Let $\G$ be a locally compact groupoid with open source and range maps, let $\H$ be a relatively
clopen subgroupoid of $\G$  and let $Z$ be a proper  $\H$-space. Then the induded $\G$-space  
$\G\times_\H Z$ is proper.
\end{proposition}

\subsection{$\G$-orders}
%\begin{lemma}\label{lemma-closed}
%Let $X$ be a locally compact  space and let $Y$ be a subset of $X$. Then $Y$ is closed if and only if  $Y\cap K$ is compact for any compact subset $K$ of $X$. \end{lemma}
%\begin{proof}
%This is clear that if $Y$ is closed, then any intersection of $Y$ with  a compact subset of $X$ is compact. Conversely, assume that  any intersection of $Y$ with a compact subset of $X$ is compact, let us prove that $X\setminus Y$ is an open set. Let $x$ be an element in $X\setminus Y$ and let $V$ be a  compact neighborhoods of $x$. When $V'$ runs through compact neighborhood of $x$,   the family  $V\cap V'\cap Y$ has empty intersection and hence there exists $V'$ a compact neighborhood of $x$ such that $V'\subseteq V$ and $V'\cap Y=\emptyset$.
%\end{proof}

\begin{definition}\label{definition-G-order}
Let $\G$ be locally compact groupoid with space of unit $X$. A $\G$-order is a subset $\r$  of $\G$ such that
\begin{itemize}
\item $u(s(\r))\subseteq \r$;
\item $\r^{-1}=\r$ ($\r$ is symmetric).
\item  for every compact subset $Y$ of $X$, then $\r_Y$ is compact.
\end{itemize}
\end{definition}
\begin{remark}\label{remark-G-order}
Let $\G$ be locally compact groupoid with unit space.
\begin{enumerate}
\item for any compact subset $K$ of $\G$, then $K\cup K^{-1}\cup r(K)\cup s(K)$ is a compact $\G$-order and hence  for any compact subset $K$ of $\G$, there exists a compact $\G$-order $\r$ such that $K\subseteq\r$.
\item if $\r_1$ and $\r_2$ are $\G$-orders, then $\r_1\cup\r_2$ and $\r_1\cap\r_2$ are $\G$-orders.
\end{enumerate}
\end{remark}
\begin{lemma}Let $\G$ be locally compact groupoid, then any $\G$-order is closed.
\end{lemma}
\begin{proof}
Let $\r$ be a $\G$-order. Let us prove that $\r\cap K$ is compact for any compact subset $K$ of $\G$.  Let us set $Y=s(K)$. Since $Y$ is compact, then $\r_Y$ is compact and hence  $K\cap \r=K\cap\r_Y$ is compact. 
\end{proof}

If $\r_1$ and $\r_2$ are two $\G$-orders, then $$\r_1\ast\r_2\defi (\r_1\cdot \r_2)\cup(\r_2\cdot \r_1)$$ is a $\G$-order. If $\r$ is a $\G$-order and
$n$ an integer, then $\r^{\ast n}$ stands for $\r\ast\cdots\ast\r$ ($n$ products). Notice that according to the first point of Definition \ref{definition-G-order}, we have that $\r\subseteq \r^{\ast n}$ for every integer $n$.
Let $\E_\G$ be the set of $\G$-orders. Then $\E_\G$ is a poset for the inclusion and ordered semi-group for $\ast$. Moreover  $\E_\G$ is a lattice with the infimium given by the intersection and the supremium given by the union. We denote by $\E_{\G,c}$ the set of compact $\G$-order. Then $\E_{\G,c}$ is  as well an ordered semi-group for $\ast$ and a lattice fo the partial order given by the inclusion.

\subsection{$\r$-decomposition of a groupoid}

\begin{remark}
Let $\G$ be  a locally compact groupoid and let $\H$ be  a relatively  clopen subgroupoid of $\G$.
\begin{enumerate}
%\item  let $Z$ be a proper subset of $\G$.  Then $Z\cap \H$ is a proper subset of $\H$.
\item let  $\r$ be a $\G$-order, then $\r\cap \H$ is a $\H$-order denoted by $\r_{/\H}$.
\item $\E_\G\to\E_\H:\r\mapsto \r_{/\H}$ is a map of posets such that  $${\r_1}_{/\H}\ast{\r_2}_{/\H}\subseteq(\r_1\ast\r_2)_{/\H}$$ for any $\G$-orders $\r_1$ and $\r_2$.
%\item the above map is splitted by $\E_\H\to\E_\G;\r\mapsto \r\cup X$, where $X$ is the unit space of $\G$.
\end{enumerate}
\end{remark}
 
%\begin{lemma}
%Let $\G$ be locally compact groupoid and let $\H$ be an open and proper subgroupoid of $\G$. Then $\H$ is locally clopen.
%\end{lemma}
%\begin{proof}
%Let $Y$ be the unit space of $\H$  and let $K$ be a compact subset of  $\G_Y^Y$.  In view of Lemma \ref{lemma-closed}, let us prove that  $\H\cap K$ is compact. Let us set 
% $Z_1=s(K)$ and $Z_2=r(K)$.
%Since $\H$ is proper and $Z_1$ and $Z_2$ are compact, then $\H_{Z_1}^{Z_2}$ is compact and hence $\H\cap K=\H_{Z_1}^{Z_2}\cap K$   is compact.\end{proof}
\begin{definition}\label{definition-remark}
Let $\G$ be a locally compact groupoid,  let $\H$ be a  subgroupoid of $\G$ with  space of units $Y$ and let $\r$ be a $\G$-order.
\begin{enumerate}
\item an $\r$-decomposition of $\H$ is a quadruple $(V_1,V_2,\H_1,\H_2)$ where 
\begin{itemize}
\item  $V_1$ and $V_2$ are open subsets of $Y$ with  $Y=V_1\cup V_2$ and such that there exists a partition of the unit subordinated to $(V_1,V_2)$;
\item $\H_1$ and $\H_2$ are subgroupoids of $\H$ which are open in $\G$.
\item $\r_{V_i}\cap\H$ is contained in $\H_i$ for $i=1,2$.
\end{itemize}
\item a coercive $\r$-decomposition of $\H$ is a $\r$-decomposition $(V_1,V_2,\H_1,\H_2)$ of $\H$ such that
$\H_1$ and $\H_2$ are relatively clopen in $\G$.
\end{enumerate}
\end{definition}
Following the route   of \cite{gwy2},  we introduce the notion of decomposability with respect to  a set of open subgroupoids.
\begin{definition}
Let $\G$ be locally compact groupoid and let $\D$ be a set of open subgroupoids of $\G$. A subgroupoid 
$\H$ of $\G$ is $\D$-decomposable if for every $\G$-order $\r$, there exists an $\r$-decomposition
$(V_1,V_2,\H_1,\H_2)$ with $\H_1$ and $\H_2$ in $\D$.
\end{definition}
\begin{remark}If the  space of units of $\G$ is second countable, then the existence of the partition of the unit in the first item of Definition \ref{definition-remark} is garanteed.\end{remark}
\begin{lemma}\label{lemma-open-decomposable}
Let $\G$ be locally compact groupoid and let $\H$ be a subgroupoid of $\G$.
\begin{enumerate}
\item if $\D$ is a  set of open subgroupoids of $\G$ such that $\H$ is  
$\D$-decomposable then $\H$ is an open subgroupoid of $\G$.
\item  if $\D$ is a  set of relatively clopen  subgroupoids of $\G$ such that $\H$ is  
$\D$-decomposable,  then $\H$ is a relatively clopen  subgroupoid of $\G$. %Moreover, if $\r$ is a  $\G$-order, every $\r$-decomposition $(V_1,V_2,\H_1,\H_2)$ of $\H$ with $\H_1$ and $\H_2$ in $\D$ is coercive.
\end{enumerate}
\end{lemma}
\begin{proof}

Let us prove the first point.
 Let $\ga$ be an element in $\H$. According to point (i) of Remark \ref{remark-G-order}, there exists a $\G$-order $\r$ such that $\ga$ lies in $\r$. Let $(V_1,V_2,\H_1,\H_2)$ be  a $\r$-decomposition of $\H$  with $\H_1$ and $\H_2$ in $\D$. By definition of an $\r$-decomposition, we see that  
$\ga$ belongs to $\H_1\cup\H_2$ which is an open subset of  $\G$  contained in $\H$.

\smallskip

 For the second point, assume now that every subgroupoid in $\D$ is  relatively clopen and let $\H$ be a $\D$-decomposable subgroupoid of $\G$.
Let us prove that   $\H$ is relatively  clopen. Let $Y$ be the unit space of $\H$. According to Lemma \ref{lemma-locally-clopen}, this amounts to prove that $\H\cap K$ is compact if $K$ is a  compact subset  of $\G_Y$.  Consider then  a $\G$-order $\r$ such that $K\subseteq \r$ (see point (i) of Remark \ref{remark-G-order}) and let $(V_1,V_2,\H_1,\H_2)$ be a $\r$-decomposition for $\H$. The existence  of a partition of the unit subordinated to $(V_1,V_2)$ ensures that there exists two closed subsets $F_1$ and $F_2$ of $Y$, respectively contained in $V_1$ and $V_2$ and such that $Y=F_1\cup F_2$. Let us set $K_1= K\cap \G_{F_1}$ and $K_2=K\cap \G_{F_2}$. Then $K_1$ and $K_2$ are  compact subsets respectively contained in $\G_{V_1}$ and $\G_{V_2}$ and moreover, we have $K=K_1\cup K_2$. 
Furthermore, since $K_1\subseteq \r_{V_1}$ and $K_2\subseteq \r_{V_2}$ and using the  definition of a $\r$-decomposition, we have $\H\cap K_1=\H_1\cap K_1$ and  $\H\cap K_2=\H_1\cap K_2$. Since $\H_1$ and $\H_2$ are relatively clopen subgroupoids, then $\H_1\cap K_1$  and 
$\H_2\cap K_2$ are compact and hence $\H\cap K$ is compact.

\end{proof} 

Let $\G$ be locally compact groupoid.  A set $\D$ of open subgroupoids of $\G$ is closed under coarse decompositions if every $\D$-decomposable  subgroupoid  of $\G$ 
  is indeed in $\D$. 
  % In the same way, a set   $\D$ of  locally clopen  subgroupoids of $\G$ is closed under coarse coercive decompositions if every coercively  $\D$-decomposable  subgroupoid  of $\G$ in $\D$.
  If $\D$ is a set of open subgroupoid of $\G$, let $\widehat{\D}$ be the smallest set of open subgroupoids of $\G$ closed under coarse decompositions. 
  
  \begin{definition}\label{definition-finite-decomposition}
  Let $\G$ be a locally compact groupoid and let $\D$ be a family of open subgroupoids of $\G$.
  An open subgroupoid $\H$ of $\G$ has  finite decomposition complexity with respect to $\D$  ($\D$-fdc) if $\H$ belongs to $\widehat{\D}$.
  \end{definition}
  
  \begin{lemma}\label{lemma-stable-open-subgroupoid}
  Let $\G$ be a locally compact groupoid and let  $\D$ be a set of open subgroupoids of $\G$ closed under taking open subgroupoids.  Then  $\widehat{\D}$ is closed under taking subgroupoids.
 \end{lemma}
 \begin{proof}
 Let $\D'$ be the set of open subgroupoids $\H$ of  $\G$ such that every  open subgroupoid of $\H$ lies in $\widehat{\D}$. We have inclusions $\D\subseteq \D'\subseteq  \widehat{\D}$. Let us show that $\D'$ is closed under coarse decompositions. Let $\H$ be an open subgroupoid of $\G$  which is $\D'$-decomposable and let $\H'$ be a open subgroupoid of $\H$ with unit space $Y$. Let $\r$ be a $\G$-order and let us consider an $\r$-decomposition                     $(V_1,V_2,\H_1,\H_2)$    of $\H$ with $\H_1$ in $\H_2$ in $\D'$. Then $(V_1\cap Y,V_2\cap Y,\H_1\cap \H',\H_2\cap \H')$ is  an $\r$-decomposition of $\H'$ with $\H_1\cap \H'$ and $\H_2\cap \H'$ in $\D$. In consequence $\H'$ is in $\widehat{\D}$ for any open subgroupoid and hence $\H$ is in $\D'$. We conclude that $\widehat{\D}\subseteq\D'$ and hence  $\widehat{\D}=\D'$.
 \end{proof}
  \begin{lemma}\label{lemma-relatively-clopen}
  Let $\G$ be a locally compact groupoid and let  $\D$ be a set of relatively  clopen  subgroupoids of $\G$.
  \begin{enumerate}
  \item if $\H$ is in  $\widehat{\D}$, then $\H$ is relatively clopen;
 % \item  $\widehat{\D}$  is the smallest set of locally clopen subgroupoids of $\G$ which is closed under   coercive coarse decomposition;
  \item If $\D$ is closed under taking relatively  clopen subgroupoids,  then so is $\widehat{\D}$.
  \end{enumerate}
 \end{lemma}  
 \begin{proof}
 To prove the first point, let us  consider the set $\D'$ of relatively clopen subgroupoids of $\G$ that belongs to $\widehat{\D}$. Then we have inclusions
 $\D\subseteq \D'\subseteq  \widehat{\D}$ and we deduce from  Lemma \ref{lemma-open-decomposable} that $\D'$ is closed under coarse decompositions. Hence we have $\D'=\widehat{\D}$
 
 \smallskip
 
 To prove the second point, we proceed as for the second point of  Lemma \ref{lemma-stable-open-subgroupoid} by considering
 the set of subgroupoids $\H$ of $\G$ for which every relatively clopen  subgroupoid  is in $\widehat{\D}$ and by noticing that the intersection of two relatively clopen subgroupoids is relatively clopen.
 \end{proof}
%   \begin{definition} Let $\G$ be a locally compact subgroupoid and let  $\D_\G$ be the  set of locally clopen proper subgroupoids of $\G$. 
%   We say that  $\G$ has proper asymptotic dynamic if $\G$ is in $\widehat{\D_\G}$. \end{definition}
%   
%   
%  \begin{remark}
%  Let   $\G$ be a  locally compact groupoid   with proper asymptotic dynamic. Since $\D_\G$ is closed under taking relatively clopen subgroupoid and in view of Corollary \ref{corollary-subgroupoid-proper} and of the first point of Lemma \ref{lemma-relatively-clopen}, we  see that  any relatively clopen subgroupoid of $\G$  has proper asymptotic dynamic.
%  \end{remark}
 \begin{example}
  Let $X$ be a metric discrete space with bounded geometry and with finite complexity decomposition in the sense of \cite{gty1} and consider then $\G_X$   the coarse groupoid of $X$ defined in \cite{sty}. Then $\G_X$ has finite decomposition  complexity with respect to the set of its compact open subgroupoids. In particular, if $\Ga$ is a finitely generated group with finite complexity decomposition and if we consider its action  on  its Stone-\v{C}ech compactification $\beta_\Ga$, then the action groupoid $\Ga\lt\beta_\Ga$ has finite decomposition  complexity with respect to the set of its compact open subgroupoids.
  \end{example}
  
%  Let  $\G$ be a  locally compact groupoid  and let $\H$ be a open subgroupoid of $\G$ with unit space $Y$. If $\r$ be a $\G$-order such that $\r_Y^Y$ is in $\H$, we  set $$\H_\r=\r_Y\cdot\H\cdot\r^Y.$$ Since $\r_Y^Y$ is in $\H$, then $\H_r$ is a subgroupoid of $\G$.
%  \begin{definition}
%  Let  $\G$ be a  locally compact groupoid and let $\H$ be a clopen subgroupoid of $\G$ with unit space $Y$
%  \begin{enumerate}
%  \item let  $\r$ be a $\G$-order such that $\r_Y^Y$ is in $\H$. An open subgroupoid $\H'$ is a $\r$-enlargement of $\H$ if $\H\subseteq \H'\subseteq  \H_\r$;
%  \item An open groupoid $\H'$ is an enlargement of $\H$ is there exist a $\G$-order $\r$ such that  $\r_Y^Y$ is in $\H$ and $\H'$ is an $\r$-expansion of $\H$.
%  \end{enumerate}
%  \end{definition}
%  
%   \begin{lemma}
%  Let $\G$ be a locally compact groupoid and let  $\D$ be a set of open (resp. relatively clopen)  subgroupoids of $\G$ closed under enlargement (resp. relatively clopen enlargement).  Then  $\widehat{\D}$ is closed under enlargement (resp. relatively clopen enlargement).
% \end{lemma}
% \begin{proof}
%  Let us  consider the set $\D'$ of  clopen subgroupoids of $\G$ for which every enlargement is in  $\widehat{\D}$. Then $\D'$ is contains $\D$ and is contained in $\D$. Let us show that $\D'$ is stable under coarse decomposition.
%  
%  
%  
%  \end{proof}
\section{Reduced crossed product of a groupoid}
In this section, we review the construction of the reduced crossed-product for a groupoid action on a $C^*$-algebra. Some good material for this construction can be founded in \cite{legall-thesis,legall-paper}.
 \subsection{$C(X)$-algebra}
  \begin{definition}
  Let $X$ be a locally compact space. A $C(X)$-algebra  is a $C^*$-algebra $A$ together with a morphism
  $\Psi:C_0(X)\to\mathcal{Z}(\mathcal{M}(A))$, where $\mathcal{Z}(\mathcal{M}(A))$ stands for
  the center of the multiplier algebra of $A$, such that
  $$\{\Psi(f)\cdot a;\, f\in C_0(X)\text{ and }a\in A\}$$ is dense in $A$.
  \end{definition}
  From now on, for $f$ in $C_0(X)$ and $a$ in $A$, we will denote $\Psi(f)\cdot a$ by $f\cdot a$ and omit  the structure map $\Phi$.
  
  \medskip
  Let $A$ be an $C(X)$-algebra and let us consider 
  for $x$ in $X$   the ideal  $I_x$  defined as the closure of 
  $$\{f\cdot a;\, f\in C_0(X)\text{ and }a\in A\text{ such that }f(x)=0\}.$$
  We define the {\bf fiber} of $A$ at $x$ as the quotient $C^*$-algebra $A_x\defi A/I_x$.
  For $a$ in $A$, we denote by $a(x)$ the image of $a$ under the quotient map
  $A\to A_x$. Then we have the following classical result \cite{wil}
  \begin{lemma}
  Let $X$ be a locally compact space and let $A$ be a $C(X)$-algebra. Then for any $a$ in $A$,
  \begin{enumerate}
  \item  the map $X\to\R;\, x\mapsto \|a(x)\|$ is upper semi-continuous and vanishing at infinity;
  \item $\| a\|=\sup_{\in X}\|a(x)\|$.
  \end{enumerate}
  \end{lemma}
  Let $X$ and $Y$ be locally compact  space, let $A$ be a $C(Y)$-algebra and let $f:X\to Y$ be a continuous map.
  The algebra $C_0(X,A)$ of continuous functions $\xi :X\to A$ vanishing at infinity is then a $C(X\times Y)$-algebra.
  Consider in $C_0(X,A)$ the ideal $I_f$ defined as the closure of
  $$\{h\cdot \xi;\, h\in C_0(X\times Y),\, \xi \in C_0(X,A)\text{ such }   h(x,f(x))=0\,\forall x\in X\}.$$
  The {\bf pull back algebra} of $A$ by $f$ is by definition $f^*A\defi C_0(X,A)/I_f$. Pointwise multiplication by $C_0(X)$ on $C_0(X,A)$ induces then a $C(X)$-algebra structure on $f^*A$. The fiber of $f^*A$ at an element $x$ of $X$ is canonically isomorphic to $A_{f(x)}$, this isomorphism being induced by the 
  map \begin{eqnarray*}C_0(X,A)&\to& A_{f(x)}\\
  \xi&\mapsto&\xi(x)(f(x)).
  \end{eqnarray*}
  Let $A$ and $B$ be two $C(X)$algebras. A morphism of $C^*$-algebra $\Psi:A\to B$ is called a morphism of $C(X)$-algebra if it is in addition $C_0(X)$-linear. It is straightforward to check that a morphism of $C(X)$-algebra $\Psi:A\to B$ induced for every $x$ in $X$ a morphism
  $\Psi_x:A_x\to B_x$. Moreover, $\Psi$ is an isomorphism (resp. injective, surjective) if  $\Psi_x$ is an isomorphism (resp. injective, surjective) for any $x$ in $X$.
 \subsection{Groupoid actions on $C^*$-algebras}
Groupoid actions generalize  to the setting of groupoid the notion of group actions by automorphisms on a $C^*$-algebra.
 \begin{definition}
 Le $\G$ be a locally compact groupoid  with $X$ as space of units and let $A$ be a $C(X)$-algebra. An action of $\G$ on $A$ is given by a $C(\G)$-isomorphism 
 $\alpha:s^*A\to r^*A$ which satisfies
 $$\al_{\ga\ga'}=\al_\ga\circ\al_\ga'$$ for any $\ga$ and $\ga'$ in $\G$ such that $s(\ga)=r(\ga')$, where $$\al_\ga:A_{s(\ga)}\to A_{r(\ga)}$$  is the morphism fiberwise induced by $\al$ at $\ga$ in $\G$ under the   canonical isomorphisms 
 $(s^*A)_\ga\cong A_{s(\ga)}$ and  $(r^*A)_\ga\cong A_{r(\ga)}$.
 A $C(X)$-algebra equipped with an action of $\G$ will be called a $\G$-algebra.
 \end{definition}
 In what follows,  for a $\G$-algebra $A$ with respect to an action
 $\al:s^*A\to r^*A$, we shall denote for short  the morphism induced fiberwise at $\ga$ in $\G$  by 
 $$\ga:A_{s(\ga)}\mapsto A_{r(\ga)}; a\mapsto \ga(a).$$
 \begin{example}\label{example-action-algebra}Let $\G$ be a locally compact groupoid with space of units $X$ and let 
 $Z$ be à $\G$-space with respect to the anchor map $p_Z:Z\to X$.
\begin{enumerate}
 \item  The  anchor map provides a $C(X)$-algebra structure on $C(Z)$ which is acted upon by $\G$ in the following way. Let us define
 $$s_*Z=\{(\ga,z)\in \G\times Z\text{ such that }  s(\ga)=p_Z(z)\}$$ and
 $$r_*Z=\{(\ga,z)\in \G\times Z\text{ such that } r(\ga)=p_Z(z)\}.$$ The we have
 canonical isomorphisms $C_0(s_*Z)\cong s^*(C_0(Z))$ and
 $C_0(r_*Z)\cong r^*(C_0(Z))$ and under these identifications, the homeomorphism
 \begin{eqnarray*}
 r_*Z&\to&s_*Z\\
 (\ga,z)&\mapsto&(\ga,\ga^{-1}z)
 \end{eqnarray*} gives rise to an $C(\G)$-isomorphism
 $$\al:s^*(C_0(Z))\stackrel{\cong}{\lto}r^*(C_0(Z)).$$ Let $\ga$ be an element in $\G$. The fibers at $\ga$ of $s^*(C_0(Z))$ and  $r^*(C_0(Z))$ are  under the above identifications respectively $C_0(Z_{s(\ga)})$ and $C_0(Z_{r(\ga)})$ and $\al$ induces fiberwise at $\ga$   the isomorphism 
 \begin{eqnarray*}
 C_0(Z_{s(\ga)})&\to&C_0(Z_{r(\ga)})\\
 f&\mapsto&\ga(f),
 \end{eqnarray*}
 where $\ga(f)(z)=f(\ga^{-1}\cdot z)$ for any $z$ in $Z_{r(\ga)}$ and any $f$ in
 $C_0(Z_{s(\ga)})$.
 \item If $A$ is a $C(Z)$-algebra, then an action of $\G\lt Z$ on $A$ is simply    an action $\al:s^*A\to r^*A$ of $\G$ on $A$ which is $C(Z)$-linear, where $A$ is viewed as a $C(X)$-algebra  by using the anchor map.
  \end{enumerate}
 \end{example}
 Let $\G$ be a locally compact groupoid with space of unit $X$ and let $A$ and $B$ be $\G$-algebras.
 A $\G$-morphism is a $C(X)$-morphism $f:A\to B$ such that $$\ga\circ f_{s(\ga)}=f_{r(\ga)}\circ \ga$$ for every $\ga$ in $\G$.

   \subsection{Reduced crossed products}
 
 Let $\G$ be a locally compact groupoid with space of units $X$ and let $C_c(\G)$ be the set of complex  valued and compactly supported continuous function on $\G$. We assume from now on that $\G$ is provided with a Haar system $(\la^x)_{x\in X}$. Let $L^2(\G)$ be 
   the $C_0(X)$-Hilbert module obtained by  completion of $C_c(\G)$ with respect to the $C_0(X)$-scalar product 
   $$\langle \eta,\eta'\rangle(x)=\int_{\G^x}\bar{\eta}(\ga^{-1})\eta'(\ga^{-1})d\la^x(\ga)$$  for any $\eta$ and $\eta'$ in $C_c(\G)$. An element  $h$ of $C_0(X)$ acts  on $L^2(\G)$ by multiplication by $h\circ s$. 
   
    %For any $\G$-order $\r$, let  $C_{c,\r}(\G)$ be the set  of functions in $C_c(\G)$ with support contained in $\r$. 
  Let $A$ be a $\G$-algebra. Recall that $r^*A\defi C_0(\G)\ts_rA$ is a $C_0(\G)$-algebra  and that for $h$ in $r^*A$ and $\ga$, then  $h(\ga)\in A_{r(\ga)}$ is  the fiber evaluation of $h$ at $\ga$ under the identification between $(s^*A)_\ga$ and $A_{s(\ga)}$. For  $h$  in $r^*A$, the support of $h$, denoted by $\supp h$,  is  the  complementary of the largest open subset of $\G$ on  which $\ga\mapsto h(\ga)$ vanishes.  Let us set then $C_c(X;\G,r^*A)$  the set of elements of $r^*A$ with compact support. In the same way, we can define $C_c(X;\G,s^*A)$ as the  set of elements of $s^*A$ with compact support. 
     
   If $A$ is a $\G$-algebra, we set $L^2(\G,A)=L^2(\G)\otimes_{s}A$. Notice that $C_c(X;\G,s^*A)$ embeds in $L^2(\G,A)$  and for any    $\eta$ and $\eta'$ in $C_c(X;\G,s^*A)$, the fiber evaluation of $\langle \eta,\eta'\rangle$ at an element $x\in X$ is the element of $A_x$ uniquely determined  by 
    $$\langle \eta,\eta'\rangle(x)=\int_{\G^x}{\eta^*}(\ga^{-1})\eta'(\ga^{-1})d\la^x(\ga).$$  
   
   Recall that  $C_c(X;\G,r^*A)$ is provided with an involutive algebra structure such that
   $$f\cdot g(\ga)=\int_{\G^{r(\ga)}}f({\ga'}){\ga'}(g({\ga'}^{-1}\ga)d\la^{r(\ga)}(\ga')$$ and
   $$f^*(\ga)=\ga(f(\ga^{-1})^*)$$ for any $f$ and $g$ in $C_c(X;\G,r^*A)$ and any  $\ga$ in $\G$.
   Moreover, for any $f$ in  $C_c(X;\G,r^*A)$,  the map    \begin{eqnarray*}
   C_c(X;\G,s^*A)&\to&C_c(X;\G,s^*A)\\
   \xi&\mapsto&f\cdot\xi
   \end{eqnarray*}
   with $$(f\cdot\xi)(\ga)=\int_{\G^{r(\ga)}}\ga^{-1}(f(\ga'))\xi({\ga'}^{-1}\ga)\la^{r(\ga)}(\ga')$$ extends
   to an adjointable endomorphism of $L^2(\G,A)$ and we obtain in this way an involutive and faithful representation of
   $C_c(X;\G,r^*A)$. The reduced crossed product algebra $A\rtr\G$ is then the closure of $C_c(X;\G,r^*A)$ in the algebra  $\mathcal{L}(L^2(\G,A))$ of adjointable endomorphisms of $L^2(\G,A)$.

   \begin{lemma}\label{lemma-scalar}Let $\G$ be a locally compact groupoid  with  space of units  $X$  provided with a Haar system. Let $V$ be an open subset and let $\phi:X\to \C$ be a bounded and continuous function  with support in $V$. Then there exists
   a bounded operator  $$\Lambda^s_\phi:A\rtr\G\to A\rtr\G$$ such that 
   \begin{enumerate}
   \item $\Lambda^s_\phi$ has operator norm bounded by $\sup_{x\in X} |\phi(x)|$;
   \item $\Lambda^s_\phi(h)=h\cdot \phi\circ s$ for all $h$ in $C_c(X;\G,A)$.
   \end{enumerate}
   \end{lemma}
   \begin{proof}
   Let us set $M=\sup_{x\in X} |\phi(x)|$. The map $$C_{c}(\G)\lto C_{c}(\G);f \mapsto f\cdot \phi\circ r$$ extends  to an adjointable operator $T_\phi:L^2(\G)\lto L^2(\G)$ such that  $\|T_\phi\|\lq M$. Then right multiplication by $T_\phi\otimes_{C_0(X)}Id_A$ on $\L(L^2(\G,A))$  preserves the subalgebra  $A\rtr\G$ and hence induces a bounded operator   $\Lambda^s_\phi:A\rtr\G\to A\rtr\G$ which satisfies the required conditions.
   \end{proof}
   \begin{remark}
   In the same way, left multiplication by $T_\phi$  on $\L(L^2(\G,A))$ preserve  $A\rtr\G$ and hence induces   
    a bounded operator   $\Lambda_\phi^r:A\rtr\G\to A\rtr\G$ such that  
    \begin{enumerate}
    \item $\Lambda^r_\phi$ has operator norm bounded by $M=sup_{x\in X} |\phi(x)|$ 
    \item $\Lambda^r_\phi(h)=h\cdot \phi\circ r$ for all $h$ in $C_c(X;\G,r^*A)$.
    \item   $\Lambda^s_{\phi}$ and $\Lambda^r_{\phi'}$ commute  for any continuous  and bounded function $\phi':X\to \C$   with support in $V$;
   \item  $\Lambda^r_{\bar{\phi}}\circ\Lambda^s_\phi: A\rtr\G\to A\rtr\G$ is positive with operator norm bounded by $M^2$;
   \item $\Lambda^r_{\bar{\phi}}\circ\Lambda^s_\phi(h)=\bar{\phi}\circ r\cdot h\cdot \phi\circ s$ for all $h$ in $C_c(X;\G,r^*A)$.
   \end{enumerate}

    \end{remark}
    
    For any open subgroupoid $\H$ of $\G$ with unit space $Y$, let $A_{/Y}$ be the closure of  $$\{f\cdot a,\, f\in C_0(Y)\text{ and } a\in A\}$$  in $A$. Then $A_{/Y}$ is a $\H$-algebra and moreover, the Haar system of $\G$ induced by restriction  a Haar system on $\H$. We  will denote the crossed product $A_{/Y}\rtr\H$ by $A\rtr \H$. Notice that since $\H$ is an open subgroupoid of $\G$, then $A\rtr \H$ can be viewed as a $C^*$-subalgebra of $A\rtr G$. 
\section{Controlled Mayer-Vietoris exact sequence in quantitative $K$-theory}
The concept of quantitative operator $K$-theory was first introduced in \cite{yuasymp} for localisation algebras in order to prove the Novikov conjecture for finitely generated groups with finite asymptotic dimension. It has been then extended in  \cite{oy2} to the setting of $C^*$-algebras equipped with a filtration arising from a length. C. Dell'Aiera developped in \cite{dell} quantitative $K$-theory in the general framework of $C^*$-algebras filtered by  abstract coarse structure.

\subsection{Review on quantitative $K$-theory}
In this subsection, we review from  \cite{dell} the main features of quantitative $K$-theory in the framework of  $C^*$-algebras filtered by an abstract coarse structure.
\begin{definition}
A coarse structure $\E$ is an ordered abelian semi-group which is  a lattice for the order. Recall that a lattice is a poset for which every pair $(E, E')$ admits a supremum $E\vee E'$ and an infimum $E\wedge E'$.
\end{definition}
\begin{example}
If $\G$ is a locally compact groupoid, then the semi-group  $(\E_\G,\ast)$ of $\G$-orders partially  ordered  by  the inclusion is a coarse structure with supremum and infimum respectively given by the union and the intersection. The same holds for the set $\E_{\G,c}$ of compact $\G$-orders.
\end{example}
\begin{definition} Let $\E$ be a coarse structure. 
A $\E$-filtered $C^*$-algebra $A$ is a $C^*$-algebra equipped with a family
$(A_E)_{E\in\E}$ of  closed linear subspaces such that:
\begin{itemize}
\item $A_E\subseteq A_{E'}$ if $E\lq E'$;
\item $A_E$ is stable by involution;
\item $A_E\cdot A_{E'}\subseteq A_{E+E'}$;
\item the subalgebra $\ds\bigcup_{E\in\E}A_E$ is dense in $A$.
\end{itemize}
Elements of $A_E$ for $E$ in $\E$ are called elements with $\E$-propagation (less than) $E$.
If $A$ is unital, we also require that the identity  $1$ is an element of $ A_E$ for every $E$ in $\E$. 
\end{definition}
Let $\E$ be a coarse structure and let $A$ and $B$ be two $\E$-filtered $C^*$-algebras.
A $C^*$-algebras homomorphism $\phi:A\to B$ is called $\E$-filtered if 
$\phi(A_E)\subseteq B_E$ for any $E$ in $\E$. 
\begin{example}Let $\G$ be a locally compact groupoid provided with a Haar system and let $A$ be a $\G$-algebra.
  For any $\G$-order $\r$, we define 
  $A\rtr\r$ as the closure in $A\rtr\G$ of the set of element $g$ in $C_c(X;\G,r^*A)$ with support in $\r$.
  Then 
  \begin{itemize}
  \item $(A\rtr\r)_{\r\in\E_\G}$ provides   $A\rtr\G$ with a structure of $\E_\G$-filtred $C^*$-algebra;
  \item  $(A\rtr\r)_{\r\in\E_{\G,c}}$ provides   $A\rtr\G$ with a structure of $\E_{\G,c}$-filtred $C^*$-algebra; 
  \item  if $\H$ is an open subgroupoid of $\G$, then $A\rtr\H$ is a $\E_\G$-filtered  $C^*$-subalgebra   of $A\rtr\H$, i.e $A\rtr \H$ is filtered by $(A\rtr\H)\cap (A\rtr\r)_{\r\in\E_\G}$;
  \item In the same way, $A\rtr\H$ is a $\E_{\G,c}$-filtered  $C^*$-subalgebra   of $A\rtr\G$.
  \end{itemize}
  Notice that if $A$ and $B$ are two $\G$-algebras and if $\phi:A\to B$ is a homomorphism of $\G$-algebras. Then the induced homomorphism
  $$\phi_\G:A\rtr\G\to B\rtr\G$$ is a $\E_{\G}$-filtered homomorphism. The same holds for 
  $\E_{\G,c}$.
\end{example}
Let  $\E$ be a coarse structure and let $A$ be  a     $\E$-filtered $C^*$-algebra.      If $A$ is not unital, let us denote by ${A^+}$ its unitarization, i.e.,
$${A^+}=\{(x,\lambda);\,x\in A\,,\lambda\in \C\}$$  with the product $$(x,\lambda)(x',\lambda')=(xx'+\lambda x'+\lambda' x,\la\la')$$ for all $(x,\lambda)$ and $(x',\lambda')$ in ${A^+}$. Then ${{A}^+}$ is $\E$-filtered with
$${{A}^+_E}=\{(x,\lambda);\,x\in {A}_{E}\,,\lambda\in \C\}$$ for any $E$ in $\E$.
We also define $\rho_A:{A^+}\to\C;\, (x,\lambda)\mapsto \lambda$.

\medbreak

Let $\E$ be a coarse structure and let $A$ be a unital $\E$-filtered $C^*$-algebra. For any  positive
number  $\eps$ with $\eps<1/4$ and any element $E$ in $\E$, we call
\begin{itemize}
\item an element $u$ in $A$  an $\eps$-$E$-unitary if $u$
  belongs to $A_E$,  $\|u^*\cdot
  u-1\|<\eps$
and  $\|u\cdot u^*-1\|<\eps$. The set of $\eps$-$E$-unitaries on $A$ will be denoted by $\operatorname{U}^{\varepsilon,E}(A)$.
\item an element $p$ in $A$   an $\eps$-$E$-projection    if $p$
  belongs to $A_E$,
  $p=p^*$ and  $\|p^2-p\|<\eps$. The set of $\eps$-$E$-projections on $A$ will be denoted by $\operatorname{P}^{\varepsilon,E}(A)$.
\end{itemize} Then $\eps$ is called the control and $E$ is called the propagation of the $\eps$-$E$-projection or of the $\eps$-$E$-unitary.
Notice that an $\eps$-$E$-unitary is invertible, and that if $p$ is an $\eps$-$E$-projection in $A$, then it has a spectral gap around $1/2$ and then gives rise by functional calculus to a  projection $\ka_{0}(p)$  in  $A$ such that
 $\|p-\ka_{0}(p)\|< 2\eps$.

 \begin{lemma}\label{cor-example-homotopy}Let $\E$ be a coarse structure and let $A$ be a unital
 $\E$-filtered $C^*$-algebra. Then for any $\eps$ in $(0,1/12)$ and any $E$ in $\E$ the following holds.
 \begin{enumerate}
 \item Let $u$ and $v$ be  $\eps$-$E$-unitaries in $A$, then $\diag(u,v)$ and
 $\diag(uv,1)$ are homotopic as   $3\eps$-$2E$-unitaries  in   $M_2(A)$; \item Let $u$ be an $\eps$-$E$-unitary in  $A$, then $\diag(u,u^*)$ and
 $I_2$ are homotopic as   $3\eps$-$2E$-unitaries  in   $M_2(A)$.
 \end{enumerate}

 \end{lemma}

 For any  $n$ integer, we set  $\ue(A)=\operatorname{U}^{\varepsilon,E}(M_n(A))$ and
$\pe(A)=\operatorname{P}^{\varepsilon,E}(M_n(A))$.
Let us  consider the inclusions
$$\P_n^{\eps,E}(A)\hookrightarrow \P_{n+1}^{\eps,E}(A);\,p\mapsto
\begin{pmatrix}p&0\\0&0\end{pmatrix}$$ and
$$\U_n^{\eps,E}(A)\hookrightarrow \U_{n+1}^{\eps,E}(A);\,u\mapsto
\begin{pmatrix}u&0\\0&1\end{pmatrix}.$$ This allows us to  define
 $$\U_{\infty}^{\eps,E}(A)=\bigcup_{n\in\N}\ue(A)$$ and
$$\P_{\infty}^{\eps,E}(A)=\bigcup_{n\in\N}\pe(A).$$

For a unital filtered $C^*$-algebra $A$, we define the
following
equivalence relations on $\P_\infty^{\eps,E}(A)\times\N$ and on  $\U_\infty^{\eps,E}(A)$:
\begin{itemize}
\item if $p$ and $q$ are elements of $\P_\infty^{\eps,R}(A)$, $l$ and
  $l'$ are positive integers, $(p,l)\sim(q,l')$ if there exists a
  positive integer $k$ and an element $h$ of
  $\P_\infty^{\eps,E}(A[0,1])$ such that $h(0)=\diag(p,I_{k+l'})$
and $h(1)=\diag(q,I_{k+l})$.
\item if $u$ and $v$ are elements of $\U_\infty^{\eps,E}(A)$, $u\sim v$ if
  there exists an element $h$ of
  $\U_\infty^{3\eps,2E}(A[0,1])$ such that $h(0)=u$
and $h(1)=v$.
\end{itemize}

If $p$ is an  element of $\P_\infty^{\eps,E}(A)$ and  $l$ is an integer, we
denote by $[p,l]_{\eps,E}$ the equivalence class of $(p,l)$ modulo  $\sim$
and if $u$ is an element of $\U_\infty^{\eps,E}(A)$ we denote by
$[u]_{\eps,E}$ its  equivalence class  modulo  $\sim$.
\begin{definition} Let $\E$ be a coarse structure, let $A$ be a $\E$-filtered $C^*$-algebra,  let $E$ be an element of $\E$ and $\eps$ be positive numbers with
  $\eps<1/4$ 
We define:
\begin{enumerate}
\item $K_0^{\eps,E}(A)=\P_\infty^{\eps,E}(A)\times\N/\sim$ for $A$ unital and
$$K_0^{\eps,E}(A)=\{[p,l]_{\eps,E}\in \P^{\eps,E}({{A^+}})\times\N/\sim \st
\rank \kappa_0(\rho_{A}(p))=l\}$$ for $A$ non unital ($\kappa_0(\rho_{A}(p))$ being the spectral projection associated to $\rho_A(p)$);
\item $K_1^{\eps,E}(A)=\U_\infty^{\eps,E}({{A}})/\sim$ if $A$ is  unital and $K_1^{\eps,E}(A)=\U_\infty^{\eps,E}({{A^+}})/\sim$ if not.
\end{enumerate}
\end{definition}

 Then $K_0^{\eps,E}(A)$ turns to be an abelian group, where
 $$[p,l]_{\eps,E}+[p',l']_{\eps,E}=[\diag(p,p'),l+l']_{\eps,E}$$  for any  $[p,l]_{\eps,E}$ and $[p',l']_{\eps,E}$ in $K_0^{\eps,E}(A)$. According to Corollary \ref{cor-example-homotopy},  $K_1^{\eps,E}(A)$ is
 equipped with a structure of abelian group such that
$$[u]_{\eps,E}+[u']_{\eps,E}=[\diag(u,v)]_{\eps,E},$$ for
any  $[u]_{\eps,E}$ and $[u']_{\eps,E}$ in $K_1^{\eps,E}(A)$.
%Nevertheless,  for any $\eps$-$r$-unitary $u$ in $M_n({A^+})$
% (with $A^+={A}$ if $A$ is already unital), we have
%  $[u]_{3\eps,2r}+[u^*]_{3\eps,2r}=0$ in $K_1^{3\eps,2r}(A)$.

%  \medskip
%
%  Recall from \cite[Corollaries 1.19 and 1.21]{oy2} that
%for any positive numbers  $r$ and $\eps$ with $\eps<1/4$, then the map
%$$K_0^{\eps,r}(\C)\to\Z;\,[p,l]_{\eps,r}\mapsto \rank\kappa_0(p)-l$$
%is an isomorphism and
%$K_1^{\eps,r}(\C)=\{0\}$.
%
% \medskip

 If $\E$ is a coarse structure, we have for any $\E$-filtered $C^*$-algebra $A$, any $E$, $E'$ in $\E$ and any positive numbers  $\eps$ and $\eps'$  with
  $\eps\lq\eps'<1/4$ and $E\lq E'$  natural group homomorphisms called the structure maps:
\begin{itemize}
\item $\iota_0^{\eps,E}:K_0^{\eps,E}(A)\lto K_0(A);\,
[p,l]_{\eps,E}\mapsto [\kappa_0(p)]-[I_l]$ (where  $\kappa_0(p)$ is the spectral projection associated to $p$);
\item $\iota_1^{\eps,E}:K_1^{\eps,E}(A)\lto K_1(A);\,
  [u]_{\eps,E}\mapsto [u]$  ;
\item $\iota_*^{\eps,E}=\iota_0^{\eps,E}\oplus \iota_1^{\eps,E}$;
\item $\iota_0^{\eps,\eps',E,E'}:K_0^{\eps,E}(A)\lto K_0^{\eps',E'}(A);\,
[p,l]_{\eps,E}\mapsto [p,l]_{\eps',E'};$
\item $\iota_1^{\eps,\eps',E,E'}:K_1^{\eps,E}(A)\lto K_1^{\eps',E'}(A);\,
  [u]_{\eps,E}\mapsto [u]_{\eps',E'}$.
\item $\iota_*^{\eps,\eps',E,E'}=\iota_0^{\eps,\eps',E,E'}\oplus\iota_1^{\eps,\eps',E,E'}$
\end{itemize}
If some of the indices $E$, $E'$ or $\eps$, $\eps'$ are equal, we shall not
repeat them  in $\iota_*^{\eps,\eps',E,E'}$.
In order to avoid overloading upperscript   in the structure maps, we shall write 
$\iota_*^{-,\eps',E'}$ for $\iota_{*}^{\eps,\eps',E,E'}$ when $\eps$ and $E$ in the source are implicit, $\iota_{*}^{\eps,E,-}$ for $\iota_{*}^{\eps,\eps',E,E'}$ when $\eps'$ and $E'$ in the range  are implicit and $\iota_{*}^{-,-}$  where  $\eps$ and $E$ in the source and
$\eps'$ and $E'$ in the range are both implicit.

\smallbreak
There is in the setting of quantitative $K$-theory the 
equivalent of the standard form. We first deal with the even case.
\begin{lemma}\label{lemma-almost-canonical-form}
Let $\E$ be a coarse structure and let  $A$  be a non unital  $\E$-filtered $C^*$-algebra. Let $\eps$ be in  $(0,\frac{1}{36})$ and let $E$ be an element in $\E$.  Then for any $x$ in
$K_0^{\eps,E}(A)$, there exist
\begin{itemize}
\item two integers $k$ and $n$ with $k\lq n$;
\item  a  $9\eps$-$E$-projection $q$  in $M_n(\widetilde{A})$
\end{itemize}
such that $\rho_A(q)=\diag(I_k,0)$ and $x=[q,k]_{9 \eps, E}$ in $K_0^{9\eps,E}(A)$.\end{lemma}
We have a similar result in the odd case.
\begin{lemma}\label{lemma-almost-canonical-form-odd}
Let $\E$ be a coarse structure and let  $A$  be a non unital  $\E$-filtered $C^*$-algebra. Let $\eps$ be in  $(0,\frac{1}{84})$ and let $E$ be an element in $\E$. 
\begin{enumerate}
\item for any $x$ in $K^{\eps,E}_1(A)$, there exists a   $21\eps$-$E$-unitary $u$ in $M_n({A}^+)$,
such that $\rho_A(u)=I_n$ and $\iota^{\eps,21\eps,E}_1(x)=[u]_{21\eps,E}$  in $K_1^{21\eps,E}(A)$;
\item if $u$ and $v$ are two  $\eps$-$E$-unitaries in  $M_n({A}^+)$ such that $$\rho_A(u)=\rho_A(v)=I_n$$ and 
$$[u]_{\eps,E}=[v]_{\eps,E}$$  in $K_1^{\eps,E}(A)$, then there exists an integer $k$ and  a homotopy $(w_t)_{t\in[0,1]}$ of
$21\eps$-$E$-unitaries of $M_{n+k}({A}^+)$ between $\diag(u,I_k)$ and $\diag(v,I_k)$ such that $\rho_A(w_t)=I_{n+k}$ for every $t$ in $[0,1]$.
\end{enumerate}
\end{lemma}

 Let  $\E$ be a coarse structure and let  $\phi:A\to B$ be  a  homomorphism  of $\E$-filtered $C^*$-algebras. Then  $\phi$ preserves $\eps$-$E$-projections
 and $\eps$-$E$-unitaries and hence $\phi$ induces  for any $E$ in $\E$ and any
$\eps\in(0,1/4)$ a
group homomorphism $$\phi_*^{\eps,E}:K_*^{\eps,E}(A)\longrightarrow
K_*^{\eps,E}(B).$$  Moreover quantitative $K$-theory is homotopy invariant with respect to homotopies which  preserve  $\E$-propagation
\cite[Lemma 1.26]{oy2}.
There is also a quantitative version of Morita equivalence \cite[Proposition 1.28]{oy2}.
\begin{proposition}\label{prop-morita} Let $\E$ be a coarse structure, let
 $A$ be  a $\E$-filtered algebra and let $\H$ be  a separable Hilbert space, then the homomorphism
$$A\to \Kp(\H)\otimes A;\,a\mapsto \begin{pmatrix}a&&\\&0&\\&&\ddots\end{pmatrix}$$
induces a ($\Z_2$-graded) group isomorphism (the Morita equivalence)
$$\MM_A^{\eps,E}:K_*^{\eps,E}(A)\to K_*^{\eps,E}(\Kp(\H)\otimes A)$$
for any $E$ in $\E$  and any
$\eps\in(0,1/4)$.
\end{proposition}

The following observation establishes a connection between quantitative $K$-theory and classical $K$-theory (see \cite[Remark 1.17]{oy2}).
\begin{proposition}\label{proposition-approximation-K-th}
Let  $\E$ be a coarse structure.
\begin{enumerate}
\item Let $A$ be a $\E$-filtered $C^*$-algebra. For any positive number $\eps$ with $\eps< \frac{1}{4}$ and  any element $y$ of $K_*(A)$,  there exists  $E$ in $\E$  and an element  $x$ of  $K_*^{\eps,R}(A)$ such that $\iota_{*}^{\eps,E}(x)=y$;
\item There exists a positive number $\lambda_0>1$ such that for any $\E$-filtered $C^*$-algebra $A$, any $E$ in $\E$, any  $\eps$  in  $(0,\ \frac{1}{4\lambda_0})$ and any element $x$ of  $K_*^{\eps,E}(A)$ for which  $\iota_*^{\eps,E}(x)=0$ in $K_*(A)$, then there exists $E'$ in $\E$  with $E'\gq E$ such that $\iota_*^{\eps,\lambda_0\eps,E,E'}(x)=0$ in $K_*^{\lambda_0\eps,E'}(A)$.
\end{enumerate}
\end{proposition}

Apply to $\G$-orders of a locally compact groupoid provided with a Haar system, we deduce the following result.
 \begin{lemma}\label{lemma-approx}
  Let $\G$ be a locally compact groupoid and let $A$ be a $\G$-algebra.
  \begin{enumerate}
  \item for every $\eps$ in $(0,1/4)$ and any $y$ in $K_*(A\rtr \G)$, there exist a compact $\G$-order $\r$  such that $\iota^{\eps,\r}_*(x)=y$.
  \item there exists $\la_0\gq 1$ such that for any $\eps$ in $(0,\frac{1}{4\la_0})$, any $\G$-order $\r$ and any $x$ in $K_*^{\eps,\r}(A\rtr\G)$ satisfying $\iota^{\eps,\r}_*(x)=0$ in $K_*(A\rtr \G)$,  there exist a $\G$-order $\r'$ with $\r\subseteq\r'$ such that $\iota^{\eps,\la_0\eps,\r,\r'}_*(x)=0$ in $K_*^{\la_0\eps,\r'}(A\rtr\G)$. The constant $\la_0$  depends neither on $A$ nor on $\G$. Moreover, if $\r$ is compact, then $\r'$ can be chosen compact.
  \end{enumerate}
  \end{lemma}

Quantitative $K$-theory inherits many features from $K$-theory. In particular there is a quantitative version of Bott periodicity and of the six-term exact sequence.

%As a consequence , we get the following approximation property.
%\begin{corollary}\label{cor-limit-quant-K-groups}
%Let $\lambda_0$ be as in Proposition \ref{proposition-approximation-K-th}. Then for any positive number $\eps$ with $\eps<\frac{1}{4\la_0}$ and for any filtered $C^*$-algebra $A$, then
%$$K_*(A)=\lim_r \iota^{\eps,\la_0\eps,r}_*(K_*^{\eps,r}(A)).$$
%\end{corollary}

\subsection{Controlled Mayer-Vietoris pair}
% associated to a $\r$-decomposition
The concept of controlled Mayer-Vietoris pair  was introduced in \cite{oy4} to streamline the 
"cut-and-pasting" technology developed by G. Yu in \cite{yuasymp} to prove the Novikov conjecture for groups with finite asymptotic dimension. It was then extended in \cite{dell} to the general setting of $C^*$-algebras filtered by a coarse structure. It gives rise to a controlled exact
sequence that allows to compute the $K$-theory by letting the propagation go to infinity.

\begin{definition}Let $\E$ be a coarse structure, 
let  $A$ be a $\E$-filtered $C^*$-algebra, let $E$ be an  element of $\E$ and let  $\Delta$ be a closed  linear subspace of $A_E$. Then a sub-$C^*$-algebra $B$ of $A$ is called an $E$-controlled  $\Delta$-neighborhood-$C^*$-algebra  if
\begin{itemize}
\item $B$ is filtered by $(B\cap A_{E'})_{E'\in\E}$;
\item  $\Delta+A_{5E}\cdot \Delta+\Delta\cdot A_{5E}+A_{5E} \cdot \Delta \cdot A_{5E}\subseteq B$.
\end{itemize}
\end{definition}

\begin{definition}
Let $\E$ be a coarse structure, 
let  $A$ be a $\E$-filtered $C^*$-algebra, let $E$ be an  element of $\E$ and let $c$ be a positive number.
 A   complete coercive decomposition pair of order  $E$ for
$A$ is a pair   $(\Delta_1,\Delta_2)$  of  closed linear subspaces of $A_E$  such that
for any $E'$ in $\E$ with $E'\lq E$, for any integer $n$ and for any $x$ in $M_n(A_{E'})$, there exists $x_1$ in $M_n(\De_1\cap A_{E'})$ and $x_2$ in $M_n(\De_2\cap A_{E'})$, both with norm at most $c\|x\|$ and such that 
$x=x_1+x_2$.  The positive  number $c$ is called the coercivity  of 
   $(\De_1,\De_2)$. 
\end{definition}

\begin{definition}\label{def-CIA} Let $S_1$ and $S_2$ be two subsets of a $C^*$-algebra $A$.
The pair $(S_1, S_2)$ is said to have  \emph{complete intersection approximation}  property (CIA) if
there exists $c>0$ such that for any positive number $\eps$, any integer $n$, any   $x\in M_n(S_1)$ and any $ y\in M_n(S_2)$  with $||x-y||<\eps $,
 there exists $z\in M_n(S_1\cap S_2)$ satisfying
$$||z-x||<c \eps, ~~~~||z-y||<c\eps.$$ The positive number $c$ is called the {\bf coercivity} of the pair 
$(S_1, S_2)$.
\end{definition}

\begin{definition}\label{definition-coarse-MV-pair}
Let  $\E$ be a coarse structure, let $A$ be a  $\E$-filtered  $C^*$-algebra  and  let $E$  be an element in $\E$ . An $E$-controlled  Mayer-Vietoris pair for $A$  is a quadruple
 $(\De_1,\De_2,A_{\Delta_1},A_{\Delta_2})$ such that for some positive number $c$.
\begin{enumerate}
\item  $(\De_1,\De_2)$ is a completely coercive  decomposition pair for $A$ of order $E$ with coercitivity $c$.
\item  $A_{\Delta_i}$ is an $E$-controlled   $\Delta_i$-neighborhood-$C^*$-algebra  for $i=1,2$;
\item the pair $ (A_{ \Delta_1, E'}, A_{\Delta_2, E'})$ has the CIA property   with coercivity $c$ for any $E'$ in $\E$.
\end{enumerate}
The positive number $c$ is called the {\bf coercivity} of the $E$-controlled Mayer-Vietoris pair $(\De_1,\De_2,A_{\Delta_1},A_{\Delta_2})$.
\end{definition}

\begin{remark}
In the above definition,
\begin{enumerate}
\item $(\Delta_1\cap A_{E'},\Delta_2\cap A_{E'},A_{\Delta_1},A_{\Delta_2})$ is an  ${E'}$-controlled Mayer-Vietoris pair for any $E'$ in $\E$ with $E'\lq E$  with same coercivity as 
$(\De_1,\De_2,A_{\Delta_1},A_{\Delta_2})$.
\item {$A_{\De_1}\cap A_{\De_2}$ is $\E$-filtered by  $(A_{\De_1,E}\cap A_{\De_2,E})_{E\in\E}$.}
\item If $A$  is a unital , we will view  ${A^+_{\Delta_1}}$  the unitarization of  ${A_{\Delta_1}}$ as
${A_{\Delta_1}}+\C\cdot 1\subseteq A$  and similarly for ${A_{\Delta_2}}$ and  ${A_{\Delta_1}}\cap {A_{\Delta_2}}$.
\end{enumerate}
\end{remark}
For purpose of rescalling the control and the propagation of an $\eps$-$E$-projection or of an $\eps$-$E$-unitary, we introduce the  following   concept of $\E$-control pair.

\begin{definition} A control pair  is a pair $(\lambda,h)$, where
\begin{itemize}
 \item $\lambda$ is a positive number  with $\lambda >1$;
\item  $h:(0,\frac{1}{4\lambda})\to \N\setminus\{0\};\, \eps\mapsto h_\eps$  is a   non-increasing map.
\end{itemize}\end{definition}
 The set of control pairs is equipped with a partial order:
 $(\lambda,h)\lq (\lambda',h')$ if $\lambda\lq\lambda'$ and $h_\eps\lq h'_\eps$
for all $\eps$ in $(0,\frac{1}{4\lambda'})$.

\begin{proposition}\label{proposition-decom-unitaries} For every positive number $c$, 
there exists a control pair $(\alpha,l)$ such that the following holds.

\medbreak

Let $\E$ be a coarse structure, let $A$  be a  unital  $\E$-filtered $C^*$-algebra, let $E$ be an element in $\E$,  let
$(\De_1,\De_2,A_{\Delta_1},A_{\Delta_2})$ be 
 controlled  Mayer-Vietoris pair for $A$ of order $E$ and coercivity $c$.
 
 \smallbreak
 
Then  for any $\eps$ in     $(0,\frac{1}{4\alpha})$ and  any       $\eps$-$E$-unitary $u$ in $A$ homotopic to $1$, there exist
a positive  integer $k$ and $w_1$ and $w_2$  two $\alpha\eps$-$l_\eps E$-unitaries in $M_k(A)$ such that
\begin{itemize}
\item $w_i-I_k$ is an  element  in the matrix algebra   $M_k(A_{\De_i})$ for $i=1,2$;
\item for $i=1,2$, there exists  a homotopy $(w_{i,t})_{t\in[0,1]}$ of  $\alpha\eps$-$l_\eps E$-unitaries between  $1$ and $w_i$ such that
$w_{i,t}-I_k\in M_k(A_{\De_i})$ for all $t$ in $[0,1]$;
\item $\|\diag(u,I_{k-1})-w_1w_2\|<\alpha\eps$.
\end{itemize}
If $A$ is a non unital $\E$-filtered $C^*$-algebra, then the same result holds for $u$ in $A^+$ such that $u-1$ is in $A$ and $u$ is homotopic to $1$ as an $\eps$-$E$-unitary in $A^+$.
\end{proposition}

\subsection{Applications to  coercive decompositions of groupoids}
In this subsection, we show that coercive decompositions of groupoids give rise to 
controlled Mayer-Vietoris pairs.

\smallbreak

      For any $\G$-order $\r$ and any open subset $V$ of the unit space of $\G$, we define 
  $A\rtr\r_V$ as the closure of set of element $h$ in $C_c(\G,A)$ with support in $\r_V$. 
   \begin{lemma}\label{lemma-coercive-decomp}Let $\G$ be a locally compact groupoid  with unit space $X$  provided with a Haar system and let $A$ be a $\G$-algebra. 
  Let $V_1$ and $V_2$ be open subsets of $X$ with $X=V_1\cup V_2$ and such that there exists a partition of the unit subordinated to $(V_1,V_2)$. Then  $(A\rtr \r_{V_1},A\rtr \r_{V_2})$ is for any $\G$-order  a coercive $\r$-decomposition pair for  $A\rtr\G$ with coercivity $1$.
  \end{lemma}
  \begin{proof}
 
  Let $(\phi_1,\phi_2)$ be a partition of the unit for $X$  subordinated to $(V_1,V_2)$. Let us consider for $i=1,2$ the bounded operator
  $$\Lambda^s_{\phi_i}:A\rtr\G\to A\rtr\G$$ of Lemma \ref{lemma-scalar} and for any $x$ in $A\rtr\r$, let us set $x_i=\Lambda^s_{\phi_i}(x)$. According to    Lemma \ref{lemma-scalar},  we have
  $x=x_1+x_2$, $\|x_i\|\lq 1$ and $x_i$ lies in $A\rtr \r_{V_i}$  for $i=1,2$.\end{proof}
  Notice that  if $\r'$ is a $\G$-order with $\r\subseteq\r'$ and $V$ is an open subset of the unit space of $\G$, then $A\rtr\r_V\subseteq A\rtr\r'_V$.

 \begin{lemma}Let $\G$ be a locally compact groupoid  with unit space $X$  provided with a Haar system. Let $\r$ and $\r'$ be $\G$-orders such that $\r'\subseteq \r$ and let $V$ be an open subset of $X$.
 Then $A\rtr\r'_V=A\rtr\r_V\cap A\rtr\r'$.
 \end{lemma}
 \begin{proof}
 We clearly have $A\rtr\r'_V\subseteq A\rtr\r_V\cap A\rtr\r'$. Conversely, let $x$ be an element in $A\rtr\r_V\cap A\rtr\r'$.
 Then there exist two sequences $(h_n)_{n\in\N}$ and $(h'_n)_{n\in\N}$ in $C_c(X;\G,r^*A)$ with support respectively in $\r_V$ and in $\r'$ converging to $x$.
 Let us set $K_n=s(\supp h_n)$ for any integer $n$ and let $\phi_n:X\to [0,1]$ be a continuous  compactly supported in $V$ and such that $\phi_n(x)=1$ for any $x$ in $K_n$. According to Lemma \ref{lemma-scalar}, we see that $$(h_n-h'_n \cdot \phi_n\circ s)_{n\in\N}=(\Lambda^s_{\phi_n}(h_n-h'_n))_{n\in\N}$$ converges to zero in $A\rtr\G$ and  hence $(h'_n \cdot \phi_n\circ s)_{n\in\N}$ is a sequence  of elements in $C_c(X;\G,r^*A)$ with support in $\r'_V$ converging to $x$.
 \end{proof}

%  \begin{lemma}Let $\G$ be a locally compact groupoid $\G$ with unit space $X$  provided with a Haar system.
%  Let $K$ be a compact subset of $\G$ and let $V_1$ and $V_2$ be open subsets of $X$ such that $X=V_1\cup V_2$.
%  Then for any $h$ in $C_c(\G,A)$ with support of $K$, there exist two elements. $h_1$ and $h_2$ of $C_c(\G,A)$ with support respectively in
%  $K\cap \G_{X\setminus V_2}$ and $K\cap\G_{X\setminus V_1}$ such that $\|h_1\|_{A\rtr\G}\lq \|h\|_{A\rtr\G}$, $\|h_2\|_{A\rtr\G}\lq \|h\|_{A\rtr\G}$ and $h=h_1+h_2$.
%  \end{lemma}
  
 As a consequence we obtain :  
 \begin{corollary}\label{corollary-complete-coercive-decomposition}
 Under the assumption of Lemma  \ref{lemma-coercive-decomp}, then  $(A\rtr \r_{V_1},A\rtr \r_{V_2})$ is for any $\G$-order $\r$ a complete 
 coercive $\r$-decomposition pair for  $A\rtr\G$ with coercivity $1$.
  \end{corollary}
  \begin{proof} Since for every integer $n$ and for $\G$-order $\r'$ such that $\r'\subseteq\r$, we have  
  $$(A\rtr\r' )\ts M_n(\C)\bigcap (A\rtr \r_{V_i})\ts M_n(\C)= (A\ts M_n(\C))\rtr \r'_{V_i}$$ for $i=1,2$, the result is a consequence of Lemma 
  \ref{lemma-coercive-decomp}.
   \end{proof}

%   \begin{lemma}Let $\G$ be a locally compact groupoid  provided with a Haar system. Let $\r$ be a $\G$-order and let 
%   $(V_1,V_2,\H_1,\H_2)$ be an $\r$-decomposition. Then we have the equality $$A\rtr \r_{/\H_1\cap\H_2}=(A\rtr\r_{/\H_1})\cap(A\rtr \r_{/\H_2}).$$
%    \end{lemma}  
%    
%    Let us prove now that if one of the subgroupoids is relatively clopen, the  conclusion of above lemma holds at any order and hence we deduce that $$A\rtr (\H_1\cap\H_2)=(A\rtr\H_1)\cap(A\rtr \H_2).$$ We start by the following preliminary lemma.
     \begin{lemma}\label{lemma-quasi-expectation}Let $\G$ be a locally compact groupoid  provided with a Haar system and let $\H$ be a relatively clopen subgroupoid of $\G$ with unit space $Y$. Then for any compactly supported continuous function $\phi : Y\to \C$ and for any $\G$-algebra $A$, there exists
     a positive continuous linear map $\Upsilon_\phi:A\rtr\G\to A\rtr\H$ such that
     \begin{enumerate}
     \item $\Upsilon_\phi(f)=\phi\circ r\cdot f_{/\H}\circ \bar{\phi}\circ s,$  for any $f$ in $C_c(X;\G,r^*A)$, where  $f_{/\H}:\H\to \C$ is the restriction of $f$ to $\H$;
  \item    $\Upsilon_\phi$ is bounded in norm by $\sup_{y\in Y}|\phi(y)|^2$.
  
  \item  $\Upsilon_\phi$ maps $A\rtr\r$ to  $A\rtr\r_{/\H}$ for any $\G$-order $\r$.
  
  \end{enumerate}
  
         \end{lemma}  
    \begin{proof}Let us denote by $\la=(\la^x)_{x\in X}$  the Haar system for $\G$. Then the restriction of $\la$ to $\H$ is a Haar system for $\H$ that we shall denote by $\la_{/\H}=(\la_{/\H}^y)_{y\in Y}$.
   Using the inclusion $C_c(\H)\hookrightarrow C_c(\G)$, we see that $L^2(\H)$ can be viewed as a $C_0(X)$-Hilbert submodule of $\L^2(\G)$ and therefore  $L^2(\H,A)$ is a right $A$-Hilbert submodule of $L^2(\G,A)$. Since $\H$ is clopen in $\G^Y$, we get that $\phi\circ r:\H\to\C$ extends to a continuous function
   $\psi:\G\to\C$ defined by $\psi(\ga)=\phi\circ r(\ga)$ if $\ga$ is in $\H$ and $\psi(\ga)=0$ else. We have $\supp \psi\subseteq \H$ and
   $|\psi(\ga)|\lq M$ for any $\ga$ in $\H$ with $M=\sup_{y\in Y}|\phi(y)|$.
   Define $T_\psi:L^2(\G)\to L^2(\G)$ as the unique bounded operator extending the map $$C_c(\G)\to C_c(\G);\xi\mapsto \psi\xi.$$ Then $T_\psi$ has operator norm bounded by $M$ and $\Im T_\psi\subseteq L^2(\H)$. In consequence $T_\psi\ts 1$ maps $L^2(\G,A)$ to $L^2(\H,A)$.
   Consider the map
   \begin{eqnarray*}
   \Upsilon_\phi:A\rtr\G&\lto&\mathcal{L}(L^2(\H,A))\\
   x&\mapsto&T_{\bar\psi}\cdot x\cdot T_\psi
   \end{eqnarray*}
   Since $T_\psi^*=T_{\bar\psi}$, we deduce that $\Upsilon_\phi$ is a positive operator with norm bounded by $M^2$.
   Moreover, for any $f$ in $C_c(X;\G,r^*A)$, any $\xi$ in $C_c(Y;\H,s^*A_{/Y})$ and any $\ga$ in $\H$, we have
   \begin{eqnarray*}
   (\Upsilon_\phi(f)\cdot\xi )(\ga)&=&\bar{\psi}(\ga)\int_{\G^{r(\ga)}}\ga^{-1}(f(\ga'))\psi(\ga'^{-1}\ga)\xi(\ga'^{-1}\ga)d\la^{r(\ga)}(\ga')\\
   &=&\bar{\psi}(\ga)\int_{\H^{r(\ga)}}\ga^{-1}(f(\ga'))\psi(\ga'^{-1}\ga)\xi(\ga'^{-1}\ga)d\la_{/\H}^{r(\ga)}(\ga')\\
   &=&\bar{\phi}\circ r (\ga)\int_{\H^{r(\ga)}}\ga^{-1}(f(\ga'))\phi\circ s(\ga')\xi(\ga'^{-1}\ga)d\la_{/\H}^{r(\ga)}(\ga')\\
    &=&\int_{\H^{r(\ga)}}\ga^{-1}(g(\ga'))\xi(\ga'^{-1}\ga)d\la_{/\H}^{r(\ga)}(\ga')
   \end{eqnarray*}
   With \begin{equation}\label{equation-support}g(\ga)=\bar{\phi}\circ r (\ga)\cdot \phi\circ s(\ga)\cdot f(\ga)=\bar{\psi}(\ga)\cdot \psi(\ga^{-1})\cdot f(\ga)\end{equation}
    for any $\ga$ in $\H$.
   Moreover $g$ has support in $\supp \psi\cap \supp f \cap  \supp^{-1} \psi$. Since $\supp \psi$ is closed in $\G$ and contained in $\H$, we deduce that $g$ is in $C_c(\H,r^*A)$.
   Hence $\Upsilon_\phi$ maps  $C_c(X;\G,r^*A)$ to $A\rtr\H$ and by continuity maps $A\rtr G$ to $A\rtr\H$. It is then clear that 
   $\Upsilon$ satisfies the required conditions.
   \end{proof}
   \begin{remark}\label{remark-support}
   According to Equation \eqref{equation-support}  and since $\H$ is open in $\G$,  we see  
    that if $f$ is  in $C_c(X;\G,r^*A)$, then $\Upsilon_\phi(f)$ is supported in $\H\cap \supp f$.
   \end{remark}
    \begin{corollary}\label{corollary-order-subgroupoid}Let $\H$ be a relatively clopen subgroupoid of a locally compact groupoid $\G$, let $\r$  be a $\G$-order and let $V$ be an open subset of $X$.
 Then we have $$(A\rtr\H)\cap (A\rtr \r)=A\rtr \r_{/\H}$$ for any $\G$-algebra $A$.
 \end{corollary}
\begin{proof}
 We clearly have $A\rtr \r_{/\H}\subseteq  (A\rtr\H)\cap (A\rtr \r)$. Conversely, let $x$ be an element in $(A\rtr\H)\cap (A\rtr \r)$.
 Then there exist two sequences $(h_n)_{n\in\N}$ and $(h'_n)_{n\in\N}$ in $C_c(X;\G,r^*A)$ with support respectively in $\H$ and in $\r$ converging to $x$.
 Let us set $K_n=s(\supp h_n)\cup r(\supp h_n) $ for any integer $n$ and let $\phi_n:X\to [0,1]$ be a continuous  function compactly supported in the unit space of $\H$ and such that $\phi_n(x)=1$ for any $x$ in $K_n$. According to Lemma \ref{lemma-quasi-expectation}, we see that $$(h_n-\Upsilon_{\phi_n}(h'_n) )_{n\in\N}=(\Upsilon_{\phi_n}(h_n-h'_n))_{n\in\N}$$ converges to zero in $A\rtr\G$. In view of Remark \ref{remark-support}, we deduce that $\Upsilon_{\phi_n}(h'_n)$ has compact support in $\r\cap\H=\r_{/\H}$, and thus 
$( \Upsilon_{\phi_n}(h'_n))_{n\in\N}$ is a sequence in $C_c(X;\G,r^*A)$ with support in $\r_{/\H}$ converging to $x$ and hence
$x$ belongs to $A\rtr \r_{/\H}$.
 \end{proof}
   \begin{corollary}\label{corollary-intersection}Let $\G$ be a locally compact groupoid  with unit space $X$  provided with a Haar system. Let $\H_1$ and $\H_2$ be relatively  clopen  subgroupoids of $\G$. Then the following holds :
   %with unit space respectively $X_1$ and $X_2$.
   \begin{enumerate}
   \item $A\rtr (\H_1\cap\H_2)=(A\rtr\H_1)\cap(A\rtr \H_2)$;
     \item $A\rtr \r_{/\H_1\cap\H_2}=(A\rtr\r_{/\H_1})\cap(A\rtr \r_{/\H_2})$ for any $\G$-order $\r$.
   \end{enumerate}
  
 \end{corollary}  
 \begin{proof}
 Let us prove the first point. We clearly have $$A\rtr (\H_1\cap\H_2)\subseteq(A\rtr{\H_1})\cap(A\rtr \H_2).$$  Conversely, let $x$ be an element in $(A\rtr{\H_1})\cap(A\rtr \H_2)$.
 Then there exist two sequences $(h_n)_{n\in\N}$ and $(h'_n)_{n\in\N}$   in $C_c(X;\G,r^*A)$ with support respectively in $\H_1$ and  
 $\H_2$ converging to $x$.
 Let us set $$K_n=s(\supp h_n)\cup r(\supp h_n)$$ for any integer $n$ and let $\phi_n:X\to [0,1]$ be a continuous function  compactly supported in the unit space of $\H_1$  and such that $\phi_n(x)=1$ for any $x$ in $K_n$. According to Lemma \ref{lemma-quasi-expectation}, we see that $$(h_n-\Upsilon_{\phi_n}(h'_n) )_{n\in\N}=(\Upsilon_{\phi_n}(h_n-h'_n))_{n\in\N}$$ converges to zero in $A\rtr\G$.  
  In view of Remark \ref{remark-support}, we deduce that  $\Upsilon_{\phi_n}(h'_n)$ has compact support in $\H_1\cap\H_2$, and thus  that
$( \Upsilon_{\phi_n}(h'_n))_{n\in\N}$ is a sequence in $C_c(X;\G,r^*A)$ with support in $\H_1\cap H_2$ converging to $x$ and hence
$x$ belongs to $A\rtr (\H_1\cap\H_2)$. To prove the second point, let us observe that according to Corollary \ref{corollary-order-subgroupoid} we have  $$A\rtr \r_{/\H_1\cap\H_2}=A\rtr \r\cap A\rtr(\H_1\cap\H_2).$$ The result is then a consequence of the first point.\end{proof}
  \begin{remark}
  The proof of the first point only requires $\H_1$ to  be relatively clopen.
  \end{remark}
  
    \begin{theorem}\label{theorem-mvpair}
  Let $\G$ be a locally compact groupoid  provided with a Haar system, let $A$ be a $\G$-algebra  and  let $\r$ and $\r'$ be  $\G$-orders such that $\r^{\ast 6 }\subseteq\r'$. Assume that   $(V_1,V_2,\H_1,\H_2)$ is  a coercive $\r'$-decomposition for $\G$. Then
  $$(A\rtr \r_{V_1},A\rtr \r_{V_2},A\rtr\H_1,A\rtr\H_2)$$ is a coercive Mayer-Vietoris pair of order $\r$  with coercivity $2$. 
  \end{theorem}
  \begin{proof}
 According to Corollary \ref{corollary-complete-coercive-decomposition}, $(A\rtr \r_{V_1},A\rtr \r_{V_2})$ is a complete coercive $\r$-decomposition pair for  $A\rtr\G$ with coercivity $1$. Let us prove that $A\rtr \H_i$ is for $i=1,2$ an $\r$-controlled  $A\rtr \r_{V_i}$-neighborhood-$C^*$-algebra. 
 By Lemma \ref{corollary-order-subgroupoid}, we see that the  $C^*$-algebra  $A\rtr \H_i$ is filtered by $$((A\rtr \H_i)\cap(A\rtr\r))_{\r\in\E_\G}=(A\rtr\r_{/\H_i})_{\r\in\E_\G}.$$
 Since $\r^{*6}\subseteq \r'$,  $\r'_{V_i}\subseteq \H_i$ and $\H_i$ is a subgroupoid of $\G$, we see that
 \begin{itemize}
 \item $\r_{V_i}\subseteq \H_i$;
 \item  $\r^{*5}\cdot \r_{V_i}\subseteq \H_i$;
 \item $\r_{V_i}\cdot \r^{*5} \subseteq \H_i$;
 \item $\r^{*5}\cdot \r_{V_i}\cdot \r^{*5}\subseteq \H_i$
 \end{itemize}
 and hence 
 \begin{itemize}
 \item  $A\rtr \r_{V_i} \subseteq A\rtr \H_i$;
 \item $A\rtr \r_{V_i}\cdot A\rtr \r^{*5} \subseteq A\rtr \H_i$;
 \item  $A\rtr \r^{*5}\cdot A\rtr \r_{V_i}\subseteq A \rtr \H_i$;
 \item $A\rtr \r^{*5}\cdot A\rtr \r_{V_i}\cdot A\rtr \r^{*5}\subseteq A \rtr \H_i$. 
 \end{itemize}
  This proves that $A\rtr \H_i$ is  an $\r$-controlled  $A\rtr \r_{V_i}$-neighborhood-$C^*$-algebra. 
 
 \smallskip
 
  Let us prove that $(A\rtr\H_1,A\rtr\H_2)$ satisfies the CIA property with coercivity $2$. Up to replace $A$ by $A\ts M_n(\C)$, it is enough to show that for every positive number and any $x_1$ in $A\rtr\r_{/\H_1}$ and $x_2$ in $A\rtr\r_{/\H_2}$ such that
 $\|x_1-x_2\|<\eps$ then there exists $z$ in $(A\rtr\r_{/\H_1})\cap (A\rtr\r_{/\H_2})$ such that  $\|z-x_1\|<\eps$. Notice that in view of  Corollary \ref{corollary-intersection}, we have $$(A\rtr\r_{/\H_1})\cap (A\rtr\r_{/\H_2})=A\rtr \r_{/(\H_1\cap\H_2)}.$$ Set $\alpha=\eps-\|x_1-x_2\|$ and let $h$ be  an element in $C_c(X;\G,r^*A)$ with support included in $\r_{/\H_1}$ and such that $\|x_1-h\|<\alpha/2$. Let $\phi:\G\to [0,1]$ be a continuous function compactly supported in the space of unit of $\H_1$ and such that
 $\phi(x)=1$ for all $x$ in $r(\supp h)\cup s(\supp h)$. According to  Lemma \ref{lemma-quasi-expectation}, we see that
 $\Upsilon_\phi(h)=h$, $\Upsilon_\phi(x_2)$ belongs to $A\rtr  \r_{/\H_1}$ and
 \begin{eqnarray*}
 \|x_1-\Upsilon_\phi(x_2)\|&<&\|x_1-h\|+\|h-\Upsilon_\phi(x_2)\|\\
 &<&\alpha/2+|\Upsilon_\phi(h)-\Upsilon_\phi(x_2)\|\\
 &<&\alpha/2+\|h-x_2\|\\
  &<&\alpha/2+\|h-x_1\|+\|x_1-x_2\|\\
  &<& \alpha/2+\alpha/2+\|x_1-x_2\|\\
  &<&\eps.
   \end{eqnarray*}
  But $x_2$ is a limit of elements of  $C_c(X;\G,r^*A)$ with support in $\r_{/\H_2}$ and hence according to Remark \ref{remark-support}, $\Upsilon_\phi(x_2)$  is also a limit of element of  $C_c(\G,r^*A)$ with support in $\r_{/\H_2}$ and therefore  $\Upsilon_\phi(x_2)$ belongs to 
  $A\rtr \r_{/\H_2}$.
  
  \end{proof}
   \begin{remark}
  In view of \cite[Theorem 6.1]{lal},  we can show in the same way that under   assumptions of Lemma  \ref{lemma-coercive-decomp},
   if $A$ is a exact and separable,
  then  $(A\rtr \r_{V_1},A\rtr \r_{V_2},A\rtr\H_1,A\rtr\H_2)$  is  a  
  $\r$-controlled  nuclear Mayer-Vietoris pair in sense of \cite[Definition 4.8]{oy4} for  $A\rtr\G$ with coercivity $2$.
 \end{remark}

%       and by
%      $$\DD_{\H_1,\H_2,*}=(\partial_{\H_1,\H_2,*}^{\eps,r})_{\eps\in (0,\frac{1}{4\al_c}),r\in (0,\frac{r}{k_{c,\eps}})}:
%      \K_*(A\rtr\H)\to \K_{*+1}(A\rtr(\H_1\cap\H_2))$$    the quantitative Mayer-Vietoris boundary map corresponding to the 
%    coercive Mayer-Vietoris   pair $(A\rtr \r_{V_1},A\rtr \r_{V_2},\A\rtr\H_1,A\rtr\H_2)$ (see ??)      
 %\end{notation}
\subsection{The Mayer-Vietoris controlled  exact sequence}\label{subsect-mv}
A coercive Mayer-Vietoris   pair gives rise to a controlled six term exact sequence that computes the quantitative $K$-theory up to the order of the  pair and up to rescaling by a control pair. In view of Theorem \ref{theorem-mvpair}, it turns out that this controlled Mayer-Vietoris  six term exact sequence was shaped out
for $K$-theory computations in the setting of coercive decompositions for  groupoids.

\begin{notation}
Let $\E$ be a coarse structure, let   $A$  be a $\E$-filtered  $C^*$-algebra, let $E$ be an element in $\E$ and let    $(\De_1,\De_2,A_{\Delta_1},A_{\Delta_2})$ be a  $E$-controlled  Mayer-Vietoris pair  for $A$. We denote by  
\begin{itemize}
\item $\jmath_{\Delta_1}: A_{\Delta_1}\to A$;
\item $\jmath_{\Delta_2}: A_{\Delta_2}\to A$;
\item $\jmath_{\Delta_1,\Delta_2}: A_{\Delta_1}\cap A_{\Delta_2} \to A_{\Delta_1}$;
\item $\jmath_{\Delta_2,\Delta_1}: A_{\Delta_1}\cap A_{\Delta_2} \to A_{\Delta_2}$
\end{itemize}
 the obvious inclusion maps.
\end{notation}

\begin{proposition}\label{prop-half-exact}For every positive number $c$, there exists a control pair $(\alpha,l)$ such that  for any coarse structure $\E$, any $\E$-filtered  $C^*$-algebra $A$, any   $E$ in $\E$ and   any   $E$-controlled   Mayer-Vietoris pair $$(\De_1,\De_2,A_{\Delta_1},A_{\Delta_2})$$    for $A$
with coercitivity $c$, then  the following holds:

\medbreak

for any $\eps$ in $(0,\frac{1}{4\al})$, for any $E'$ in $\E$ such that $l_\eps \cdot E'\lq E$, for any $y_1$ in $K^{\eps,E'}_*(A_{\Delta_1})$ and any $y_2$ in $K^{\eps,E'}_*(A_{\Delta_2})$
such that $$\jmath^{\eps,E'}_{\Delta_1,*}(y_1)=\jmath^{\eps,E'}_{\Delta_,*}(y_2)$$ in 
$K^{\eps,E'}_*(A)$, then there exists an element $x$ in 
$K^{\al \eps,l_\eps E'}_*(A_{\Delta_1}\cap A_{\Delta_2})$ such 
that
$$\jmath^{\al\eps,l_\eps E'}_{\Delta_1,\Delta_2,*}(x)=\iota_*^{-,\al\eps,l_\eps E'}(y_1)$$ in
$K^{\al \eps,l_\eps E'}_*(A_{\Delta_1})$ and
$$\jmath^{\al\eps,l_\eps E'}_{\Delta_2,\Delta_1,*}(x)=\iota_*^{-,\al\eps,l_\eps E'}(y_2)$$ in
$K^{\al \eps,l_\eps E'}_*(A_{\Delta_2})$.
\end{proposition}
In other word, this means that the composition
$$K_*^{\bullet,\bullet}(A_{\Delta_1}\cap A_{\Delta_2})
\stackrel{(\jmath^{\bullet,\bullet}_{\Delta_1,\Delta_2,*},\jmath^{\bullet,\bullet}_{\Delta_2,\Delta_1,*})}{-\!\!\!-\!\!\!-\!\!\!-\!\!\!-\!\!\!-\!\!\!-\!\!\!-\!\!\!-\!\!\!-\!\!\!-\!\!\!-\!\!\!\lto}
K^{\bullet,\bullet}_*(A_{\Delta_1})\oplus K^{\bullet,\bullet}_*(A_{\Delta_2})
\stackrel{(\jmath^{\bullet,\bullet}_{\Delta_1,*}-\jmath^{\bullet,\bullet}_{\Delta_2,*})}{-\!\!\!-\!\!\!-\!\!\!-\!\!\!-\!\!\!-\!\!\!-\!\!\!-\!\!\!\lto}K^{\bullet,\bullet}_*(A)$$ is exact at order $E$, up to rescaling by $(\al,l)$. We shall see later on that this composition can be completed  at order $E$,   up to rescaling by a control pair, in a  six term exact sequence (called in \cite[Theorem 8.4]{bd}  
 $E$-controlled Mayer-Vietoris exact sequence).

We introduce first the quantitative boundary map that fits into the controlled Mayer-Vietoris exact.
\begin{lemma}\label{lemma-technic}For every positive number $c$, there exists a control pair $(\lambda,k)$ such that the following holds:

\smallskip

Let $\E$ be a coarse structure, let $A$ be a  unital  $\E$-filtered $C^*$-algebra,
let $E$ be an element in $\E$  and  let $(\De_1,\De_2,A_{\Delta_1},A_{\Delta_2})$ be a   $E$-controlled Mayer-Vietoris pair  for $A$   with coercitivity $c$. Let   $E'$ be an element in $\E$ such that $2\cdot E'\lq E$,  let $\eps$ be  in   $(0,\frac{1}{4\lambda^3})$, let  $m$ and $n$ be integers and let  $u$ be in $U_n^{\eps,E'}(A)$, let  $v$ be  in $U_m^{\eps,E'}(A)$ and let $w_1,\,w_2$ be $\eps$-$E'$-unitaries  in $M_{n+m}(A)$ such that
\begin{itemize}
\item $w_i-I_{n+m}$ is an element  in the matrix algebra   $M_{n+m}({A_{\Delta_i}})$ for $i=1,2$;
\item $\|\diag(u,v)-w_1w_2\|<\eps$.
\end{itemize}
Then,
\begin{enumerate} 
\item there exists  a $\lambda\eps$-$k_\eps  E'$-projection $q$ in $M_{n+m}(A)$ such that
\begin{itemize}
\item $q-\diag (I_n,0)$  is an element  in the matrix algebra   $M_{n+m}(A_{\Delta_{1}}\cap A_{\Delta_{2}})$;
\item $\|q-w_1^*\diag (I_n,0)w_1\|<\lambda\eps$;
\item $\|q-w_2\diag (I_n,0)w_2^*\|<\lambda\eps$.
\end{itemize}
\item if $q$ and $q'$ are two $\lambda\eps$-$k_\eps  E'$-projections   in $M_{n+m}(A)$  that satisfy the first  point, then $$[q,n]_{\lambda^2\eps,k_\eps E'}=[q',n]_{\lambda^2\eps,k_\eps E'}$$ in $K^{\lambda^2\eps,k_\eps  E'}_0(A_{\De_1}\cap A_{\De_2})$.
\item Let  $(w_1,w_2)$ and $(w'_1,w'_2)$ be two pairs of   $\eps$-$E'$-unitaries  in $M_{n+m}^{\eps,E'}(A)$ satisfying the assumption of the lemma and let $q$ and $q'$ be  $\lambda\eps$-$k_\eps \cdot E'$-projections  in $M_{n+m}(A)$  that satisfy the first point relatively to  respectively  $(w_1,w_2)$ and $(w'_1,w'_2)$, then $$[q,n]_{\lambda^3\eps, 2k_\eps   E'}=[q',n]_{\lambda^3\eps,2k_\eps  E'}$$ in $K^{\lambda^3\eps,2k_\eps  E'}_0(A_{\De_1}\cap A_{\De_2})$. \end{enumerate}
\end{lemma}

\begin{remark}\label{remark-unicite-lemme-technic}\
We have a similar statement in the non-unital case with  $u$ in $U_n^{\eps,E'}({A^+})$ and  $v$ in $U_m^{\eps,E'}({A^+)}$  such that $u-I_n$ and $v-I_m$ have coefficients in $A$.
\end{remark}

We recall now the definition of  the quantitative boundary map associated to a controlled Mayer-Vietoris pair.
Let $\E$ be a coarse structure, let $A$ be a  $\E$-filtered   $C^*$-algebra and  let $(\De_1,\De_2,A_{\Delta_1},A_{\Delta_2})$ be a  $E$-controlled  Mayer-Vietoris pair  for $A$ with coercitivity $c$. Assume first that $A$ is unital.

Let $(\alpha,l)$ be a control pair as is Proposition \ref{proposition-decom-unitaries}. For any  $\eps$ in  $(0,\frac{1}{4\alpha})$, any $E'$ in $\E$ such that $2E'\lq E$   and any $\eps$-$E'$-unitary $u$ in $\M_n(A)$, let $v$ be an $\eps$-$E'$-unitary in some $\M_m(A)$ such that
$\diag(u,v)$ is homotopic to $I_{n+m}$ as a $3\eps$-$2E'$-unitary in  $\M_{n+m}(A)$, we can take for instance $v=u^*$ (see Lemma \ref{cor-example-homotopy}). According to Proposition \ref{proposition-decom-unitaries} and up to replacing $v$ by $\diag(v,I_k)$ for some integer $k$, there exists   $w_1$ and $w_2$ two $3\alpha\eps$-$2l_{3\eps} E'$-unitaries  in $M_{n+m}(A)$  such that
\begin{itemize}
\item $w_i-I_{n+m}$ is an element  in the matrix algebra  $M_{n+m}(A_{\Delta_i})$ for $i=1,2$;
\item for $i=1,2$, there exists  a homotopy $(w_{i,t})_{t\in[0,1]}$ of  $3\alpha\eps$-$2l_{3\eps}E'$-unitaries between  $1$ and $w_i$ such that
$w_{i,t}-I_{n+m}$ is an element  in the matrix algebra   $M_{n+m}(A_{\De_i})$ for all $t$ in $[0,1]$.
\item $\|\diag(u,v)-w_1w_2\|<3\alpha\eps$.
\end{itemize}
Let $(\lambda,k)$ be the control pair of  Lemma \ref{lemma-technic} (recall that $(\la,k)$  depends only on the coercivity  $c$). Then if $\eps$ is in $(0,\frac{1}{12\al\la^3})$,  there exists
   a $3\alpha\lambda\eps$-$2l_{3\eps} k_{3\alpha\eps} E'$-projection $q$ in $M_{n+m}(A)$ such that
\begin{itemize}
\item $q-\diag (I_n,0)$ is an element  in the matrix algebra    $$M_{n+m}(A_{\Delta_{1}\cap A_{\Delta_{2}}});$$
\item $\|q-w_1^*\diag (I_n,0)w_1\|<3\alpha\lambda\eps$;
\item $\|q-w_2\diag (I_n,0)w_2^*\|<3\alpha\lambda\eps$.
\end{itemize}
In view of the second point of Lemma   \ref{lemma-technic},  the class $[q,n]_{3\alpha\lambda^3\eps,4l_{3\eps} k_{3\alpha\eps} E'}$ in
$$K_0^{3\alpha\lambda^3\eps,4l_{3\eps} k_{3\alpha\eps} E'}(A_{\Delta_1}\cap A_{\Delta_2})$$ does not depend on the choice of $w_1$, $w_2$ or $q$.
Set then $\alpha_c=3\alpha\lambda^3$ and $$k_c:\left(0,\frac{1}{4\alpha_c}\right)\lto \N\setminus\{0\},\,\eps\mapsto 4l_{3\eps} k_{3\alpha\eps}$$ and define
$\partial^{\eps,E'}_{\Delta_1,\Delta_2,1}([u]_{\eps,E'})=[q,n]_{\alpha_c\eps,k_c E'}$. Then  for any $\eps$ in $(0,\frac{1}{4\al_c})$ and any $E'$ in $\E$ such that $k_{c,\eps}E'\lq E$, the morphism
$$\partial^{\eps,E'}_{\Delta_1,\Delta_2,1}:K_1^{\eps,E'}(A)\to K_0^{\alpha_c\eps,k_c E'}(A_{\Delta_1}\cap A_{\Delta_2})$$ is well defined.

In the non unital case $\partial^{\eps,s}_{\Delta_1,\Delta_2,1}$ is defined similarly by  noticing that in view of Lemma \ref{lemma-almost-canonical-form-odd} and up to rescaling  $\eps$, every element $x$ in $K^{\eps,E}_1(A)$ is the class of an $\eps$-$E$-unitary $u$  in some $M_n({A}^+)$ such that $u-I_n$ has coefficients in $A$.
It is straightforward to check that
$\partial^{\bullet,\bullet}_{\Delta_1,\Delta_2,1}$ is compatible with the structure morphisms, i.e $$\iota_*^{-,\al_c\eps'',k_{c,\eps''}E''}\circ \partial^{\eps,E'}_{\Delta_1,\Delta_2,1}=
\partial^{\eps'',E''}_{\Delta_1,\Delta_2,1}\circ \iota_*^{\eps',E',-}$$ for any 
$\eps'$ and $\eps''$  in $(0,\frac{1}{4\al_c})$ and any $E'$ and $E''$ in $\E$ with $E'\lq E''$ and $k_{c,\eps'}E'\lq k_{c,\eps''}E''\lq E$.

In the even case, the quantitative boundary map associated to a controlled Mayer-Vietoris pair is defined by using controlled Bott periodicity \cite[Section 2]{dell}. Up to rescale the control pair $(\al_c,k_c)$, we obtain
for any $\eps$ in $(0,\frac{1}{4\al_c})$ and any $E'$ in $\E$ such that $k_{c,\eps}E'\lq E$, the morphism
$$\partial^{\eps,E'}_{\Delta_1,\Delta_2,0}:K_0^{\eps,E'}(A)\to K_1^{\alpha_c\eps,k_c E'}(A_{\Delta_1}\cap A_{\Delta_2}).$$
We set then $$\partial^{\eps,E'}_{\Delta_1,\Delta_2,*}=\partial^{\eps,E'}_{\Delta_1,\Delta_2,0}\oplus \partial^{\eps,E'}_{\Delta_1,\Delta_2,1}.$$
Then $$\partial^{\eps,E'}_{\Delta_1,\Delta_2,*}:K_*^{\eps,E'}(A)\to K_{*+1}^{\alpha_c\eps,k_c E'}(A_{\Delta_1}\cap A_{\Delta_2})$$ is a morphism of degree $1$ compatible with the structure morphisms called the $\eps$-$E'$-quantitative Mayer-Vietoris boundary map.

\smallbreak

Notice that the  quantitative boundary map associated to an $E$-controlled  Mayer-Vietoris pair   is natural in the following sense:  let $\E$ be a coarse structure, let $A$ and $B$ be $\E$-filtered $C^*$-algebras, let $E$ be an element in $\E$, let  $(\De_1,\De_2,A_{\Delta_1},A_{\Delta_2})$  and 
$(\De'_1,\De'_2,B_{\Delta'_1},B_{\Delta'_2})$ be respectively 
 $E$-controlled  Mayer-Vietoris pairs  for $A$ and $B$ with   coercivity $c$    and let $f:A\to B$ be a   homomorphism of $\E$-filterered $C^*$-algebras such that $f(\De_1)\subseteq \De'_1,\,f(\De_2)\subseteq \De'_2,\,f(A_{\De_1})\subseteq B_{\De'_1}$ and$f(A_{\De_2})\subseteq B_{\De'_2}$. Then we have
  \begin{equation}\label{equ-boundary-natural}
f^{\al_c\eps,k_{c,\eps}E'}_{/A_{\De_1}\cap A_{\De_2},*} \circ \partial^{\eps,E'}_{\De_1,\De_2,*}=\partial^{\eps,E'}_{\De'_1,\De'_2,*} \circ f^{\eps,E'}_*,\end{equation} for any $\eps$ in $(0,\frac{1}{4\al_c})$ and any $E'$ in $\E$ with $k_{c,\eps}E'\lq E$, 
where $$f_{/A_{\De_1}\cap A_{\De_2}}:A_{\De_1}\cap A_{\De_2}\to B_{\De'_1}\cap B_{\De'_2}$$ is the restriction of $f$ to $A_{\De_1}\cap A_{\De_2}$.

We now investigate  controlled  exactness at the source for  the quantitative boundary map associated to a controlled  Mayer-Vietoris pair. 
We start with the following lemma which will play a key role in the proof of the main theorem.

\begin{lemma}\label{lemma-bound1}
There exists a control pair $(\lambda,l)$ such that 
\begin{itemize}
\item for any coarse structure $\E$, any  unital $\E$-filtered $C^*$-algebra $A$ and any subalgebras  $A_1$ and $A_2$   of $A$ such that $A_1,\,A_2$  and $A_1\cap A_2$ are respectively filtered by     $(A_1\cap A_E)_{E\in\E},\,(A_2\cap A_E)_{E\in\E}$ and $(A_1\cap A_2\cap A_E)_{E\in\E}$;
%\item the pair $(A_1\cap A_r,A_2\cap A_r)$ satisfy the CIA property with coercitivity $c$ for any positive number $r$;
\item for any positive number $\eps$  with  $\eps<\frac{1}{4\lambda}$, any $E$ in $\E$,   any integers $n$ and $m$ and any $\eps$-$E$-unitaries $u_1$ in $M_n(A)$ and $u_2$ in $M_m(A)$;
\item for any $\eps$-$E$-unitaries $v_1$ and $v_2$ respectively in $M_{n+m}({A^+_1})$ and 
$M_{n+m}({A^+_2})$ such that  
\begin{itemize} 
\item $\|\diag(u_1,u_2)-v_1v_2\|<\eps$;
\item there exists an $\eps$-$E$-projection  $q$ in $M_{n+m}(A)$ such that $q-\diag(I_n,0)$ is in $M_{n+m}(A_{1}\cap A_{2}),\,\|q-v_1^*\diag(I_n,0) v_1\|<\eps$ and $[q,n]_{\eps,E}=0$ in $K^{\eps,E}_0(A_{1}\cap A_{2})$.
\end{itemize}
\end{itemize}
Then there exists an integer $k$ and  $\lambda \eps$-$l_{\eps} E$-unitaries $w_1$ and $w_2$ respectively in  $M_{n+k}({A^+_{1}})$ and $M_{n+k}({A^+_{2}})$ such that
 $$\|\diag(u_1,I_k)-\diag(w_1w_2)\|<\la \eps.$$
 Moreover, if $v_i-I_{n+k}$ lies in $M_{n+k}(A_i)$ for $i=1,2$ then $w_1$ and $w_2$ can be chosen such that
  $w_i-I_{n+k}$ lies in $M_{n+m}(A_i)$ for $i=1,2$.\end{lemma}
As a consequence,  we deduce   controlled  exactness at the source for  the quantitative boundary map associated to a controlled  Mayer-Vietoris pair. Moreover,   this controlled  exactness is persistent at any order. 

\begin{corollary}
%\label{prop-bound1}
For any positive number $c$, there exists a control pair $(\lambda,l)$ such that 
\begin{itemize}
\item for any  coarse structure $\E$ and any $\E$-filtered $C^*$-algebra $A$;
\item for any $E$ in $\E$ and any    $E$-controlled Mayer-Vietoris pair $(\De_1,\De_2,A_{\Delta_1},A_{\Delta_2})$  for $A$  with coercitivity $c$.
\item for any positive numbers $\eps'$ and $\eps''$  with   $0<\al_c\eps'\lq \eps''<\frac{1}{4\lambda}$ and any $E'$ and $E''$ in $\E$ with $ k_{c,\eps'}E'\lq E$ and
$k_{c,\eps'}E' \lq E''$
\end{itemize}
then for any $y$ in $K_*^{\eps,E'}(A)$ such that $$\iota_*^{-,\eps'',E''}\circ\partial^{\eps',E'} _{\De_1,\De_2,*}(y)=0$$ in 
$K_{*+1}^{\eps'',E''}(A_{\De_1}\cap A_{\De_2})$,  there exist $x_1$ in $K_{*}^{\lambda\eps'',l_{\eps'' }E''}(A_{\De_1})$ and 
 $x_2$ in $K_*^{\lambda\eps'',l_{\eps''}E''}(A_{\De_2})$ such that 
$$\iota_*^{-, \lambda\eps'',l_{\eps''}E''}(y)=\jmath^{\lambda\eps'',l_{\eps''}E''}_{\Delta_1,*}(x_1)-\jmath^{\lambda\eps'',l_{\eps'' }E''}_{\Delta_2,*}(x_2).$$\end{corollary}

We us now investigate the controlled  exactness at the range for  the quantitative boundary map associated to a controlled  Mayer-Vietoris pair. 
We start with the following lemma which will play also a key role in the proof of the main theorem.
\begin{lemma}\label{lemma-exactness-inclusion}There exists a control pair $(\lambda,h)$ such that the following holds:

\smallskip

\begin{itemize}
\item Let $\E$ be a coarse structure, let $A$ be a unital $\E$-filtered $C^*$-algebra and let $A_1$ and $A_2$ be subalgebras of $A$ such that
$A_1,\,A_2$ and $A_1\cap A_2$ are respectively filtered by     $(A_1\cap A_E)_{E\in\E}$, 
$(A_2\cap A_E)_{E\in \E}$ and
$(A_1\cap A_2\cap A_E)_{E\in\E}$;
\item let $\eps$ be in $(0,\frac{1}{4\la})$ and let $E$ be in $\E$;
\item  let $n$ and $N$ be positive integers with $n\lq N$ and let $p$ an $\eps$-$E$-projection in $M_N(({A_1\cap A_2})^+)$ such that $\rho_{A_1\cap A_2}(p)=\diag(I_n,0)$.
\end{itemize}
Assume 
that
 \begin{itemize}
 \item $p$ is homotopic to $\diag(I_n,0)$ as an $\eps$-$E$-projection in
 $M_N({A^+_{1}})$;
 \item $p$ is  homotopic to $\diag(I_n,0)$ as an $\eps$-$E$-projection in
 $M_N({A^+_{2}})$.
 \end{itemize}Then there exist  an integer $N'$ with $N'\gq N$,  and four  $\la\eps$-$h_\eps E$-unitaries $w_1$ and $w_2$ in $M_{N'}(A)$,  $u$ in $M_n(A)$ and $v$ in $M_{N'-n}(A)$  such that
\begin{itemize}
\item $w_i-I_{N'}$ is an element  in   $M_{N'}(A_{i})$ for $i=1,2$;
\item $$\|w_1^*\diag(I_n,0)w_1-diag(p,0)\|<\la\eps$$ and
$$\|w_2\diag(I_n,0)w_2^*-\diag(p,0)\|<\la\eps.$$
 \item for $i=1,2$, then $w_i$ is connected to $I_{N'}$ by a homotopy
of $\la\eps$-$h_\eps E$-unitaries  $(w_{i,t})_{t\in[0,1]}$ in $M_{N'}(A)$ such  that $w_{i,t}-I_{N'}$ is in $M_{N'}(A_{i})$ for all $t$ in $[0,1]$. 
\item $\|\diag(u,v)-w_1w_2\|<\la\eps$.
\end{itemize}
\end{lemma}

As a consequence,  we deduce   controlled  exactness at the range for  the quantitative boundary map associated to a controlled  Mayer-Vietoris pair.

\begin{proposition}\label{prop-bound2}For every positive number $c$, there exists a control pair $(\al,l)$ such that  for any coarse structure $\E$, any $\E$-filtered  $C^*$-algebra $A$,  any $E$ in $\E$   and
any $E$-controlled   Mayer-Vietoris pair $ (\De_1,\De_2,A_{\Delta_1},A_{\Delta_2})$  for $A$   with coercitivity $c$, then the 
 following holds:
 
 \smallbreak
 for any $\eps$ in $(0,\frac{1}{4\la\al_c})$ and any $E'$ in $\E$ with $k_{c,\la\eps} l_\eps E'\lq E$, any
 $y$ in $K_{*}^{\eps,E'}(A_{\De_1}\cap A_{\De_2})$ such that
 $$\jmath^{\eps,E'}_{\Delta_1,\De_2,*}(y)=0$$ in $K_{*}^{\eps,E'}(A_{\De_1})$ and
 $$\jmath^{\eps,E'}_{\Delta_2,\De_1,*}(y)=0$$ in $K_{*}^{\eps,E'}(A_{\De_2})$,
 there exists an element $x$ in $K_{*+1}^{\la\eps,l_\eps E'}(A)$ such
 that
 $$\partial^{\la\eps,l_\eps E'}_{\De_1,\De_2,*}(x)=
 \iota_*^{-,\al_c\la\eps,k_{c,\la\eps} l_\eps E'}(y)$$ in
 $K_{*}^{\al_c\la\eps,k_{c,\la\eps} l_\eps E'}(A_{\De_1}\cap A_{\De_2})$.
\end{proposition}

  \begin{example}
  With notations of Theorem \ref{theorem-mvpair}, we denote by
      $$\jmath_{1,2,A}:A\rtr(\H_1\cap \H_2)\hookrightarrow  A\rtr\H_1$$
      $$\jmath_{2,1,A}:A\rtr(\H_1\cap \H_2)\hookrightarrow A\rtr\H_2$$      
      $$\jmath_{1,A}:A\rtr \H_1\hookrightarrow A\rtr\H$$
      $$\jmath_{2,A}:A\rtr \H_2 \hookrightarrow A\rtr\H$$   the obvious inclusions
      and  for any $\eps$ in $(0,1)$ and any  $\r_0$ in $\E$ such that $k_{c,\eps}\r_0\subseteq \r$, we denote  by 
      $$\partial_{\H_1,\H_2,A,*}^{\eps,\r_0}=\partial_{A\rtr \H_1,A\rtr \H_2,*}^{\eps,\r_0}:
      K_1^{\eps,\r_0}(A\rtr\H)\to K_0^{\al_c\eps,k_{c,\eps}\r_0}(A\rtr(\H_1\cap\H_2))$$ 
      $\eps$-$\r_0$-quantitative Mayer-Vietoris boundary map associated to the 
     $\r_0$-controlled  Mayer-Vietoris pair        $$(A\rtr \r_{V_1},A\rtr\r_{V_2},A\rtr\H_1,A\rtr\H_2).$$ \end{example}

\section{Statement of the main result}
  \begin{theorem}\label{theorem-main}
  Let $\G$  be a  locally  compact groupoid  provided with a Haar system 
  and let $f:A\to B$ be a homomorphism of $\G$-algebras. Let us assume that there exists a  subset $\D$ of relatively clopen groupoids of $\G$, closed under taking relatively  clopen subgroupoids and  such that
  \begin{enumerate}
  \item $f_{\H,*}:K_*(A\rtr\H)\to K_*(B\rtr\H)$ is an isomorphism for any $\H$ in $\D$  
    \item $\G$ has finite $\D$-complexity.
  \end{enumerate}
  Then  $$f_{\G,*}:K_*(A\rtr\G)\to K_*(B\rtr\G)$$  is an isomorphism.
  \end{theorem}
   
  Theorem \ref{theorem-main} is a consequence of  the following result:
  
   \begin{lemma}
  Let $\G$  be a  locally   groupoid  provided with a Haar system, let $\H$ be a relatively clopen subgroupoid of $\G$,
  let $f:A\to B$ be a homomorphism of $\G$-algebras. Let us assume that there exists a  subset $\D$ of relatively clopen groupoids of $\G$, closed under taking relatively  clopen subgroupoids and  such that
  \begin{enumerate}
  \item $f_{\H',*}:K_*(A\rtr\H')\to K_*(B\rtr\H')$ is an isomorphism for any $\H'$ in $\D$.  
    \item $\H$ is $\D$-decomposable.
  \end{enumerate}
  Then  $$f_{\H,*}:K_*(A\rtr\H)\to K_*(B\rtr\H)$$  is an isomorphism.
  \end{lemma}  
  \begin{proof}
  By Bott periodicity, this amounts to prove   that
   $$f_{\H,*}:K_1(A\rtr\H)\to K_1(B\rtr\H)$$  is an isomorphism. Let $$\widetilde{f_\H}:\widetilde{A\rtr\H}\lto \widetilde{B\rtr\H}$$ be the unitarisation of
     $f_\H$ with  $\widetilde{f_\H}$, $\widetilde{A\rtr\H}$ and  $\widetilde{B\rtr\H}$ respectively equal to 
      ${f_\H}$, ${A\rtr\H}$ and  ${B\rtr\H}$     if $f_\H$ is already a morphism of unital $C^*$-algebras and   ${f^+_\H}$, ${(A\rtr\H})^+$ and  $({B\rtr\H})^+$             else.      Let us fix  a control pair $(\la,l)$ such that 
      \begin{itemize}
      \item $\la\geqslant \la_0$, where $\la_0$ is the constant of Lemma \ref{lemma-approx}; 
      \item $(\la,l)$ is larger than \begin{itemize}
      \item the control pair corresponding to the quantitative boundary map associated to a coarse  Mayer-Vietoris pair with coercitivity $c=2$ (see Section \ref{subsect-mv}), 
      \item the control pairs of Proposition \ref{proposition-decom-unitaries} and of  Lemmas \ref{lemma-bound1}  and   \ref{lemma-exactness-inclusion}.\end{itemize}
      \end{itemize}
           We proceed by using a quantitative version of the five lemma.
     
     \smallskip 
     
\subsection{Injectivity part}   Let $x$ be an element in $K_1(A\rtr \H)$ such that $f_{\H,*}(x)=0$ in 
     $K_1(B\rtr \H)$.  Let us show then that $x=0$. We divide the proof  in five steps.\subsubsection*{Step I} Let us fix a positive number $\eps$ in $(0,\frac{1}{256\la^5})$.
      According to Lemma \ref{lemma-approx}, there exist up to stabilisation a $\G$-order $\r_0$ and an $\eps$-$\r_0$-unitary in $1+{A\rtr\H}$ such that
     \begin{itemize}
     \item $\iota_*^{\eps,\r_0}([u]_{\eps,\r_0})=x$;
     \item  $[\widetilde{f_\H}(u)]_{\eps,\r_0}=0$ in $K_1^{\eps,\r_0}(B\rtr H)$.
     \end{itemize}
     Let $(V_1,V_2,\H_1,\H_2)$ be 
     a  $(6l_\eps\cdot \r_0)$-decomposition of $\H$ with $\H_1$ and $\H_2$ in $\D$. In view of Theorem \ref{theorem-mvpair}, we see that
     $$(A\rtr \r_{V_1},A\rtr\r_{V_2},A\rtr\H_1,A\rtr\H_2)$$ is a $ l_\eps\cdot \r_0$-controlled  Mayer-Vietoris pair  
       relatively to $A\rtr\H$ and 
      $$(B\rtr \r_{V_1},B\rtr\r_{V_2},B\rtr\H_1,B\rtr\H_2)$$ is quantitative Mayer-Vietoris pair of order
      $l_\eps\cdot\r_0$ relatively to $B\rtr\H$.   For seek of simplicity, we rescale $(\al_c,k_c)$ to be equal to  $(\la,l)$.     
      
                According to Proposition \ref{proposition-decom-unitaries} applied with coercity $c=2$, there exists up to stabilisation two
       $\la\eps$-$(l_{\eps}\cdot \r_0)$-unitaries $v_1$ and $v_2$ in $\widetilde{B\rtr\H}$ such that
       
       \begin{itemize}
       \item $v_i-1$ is in $B\rtr\H_i$ for $i=1,2$;
       \item $v_i$ is homotopic to $1$ as an $\la\eps$-$(l_{\eps}\cdot \r_0)$-unitaries in   $\widetilde{B\rtr\H_i}$  for $i=1,2$;
      \item $\|\widetilde{f_\H}(u)-v_1v_2\|\lq\la\eps$.
       \end{itemize}
       
\subsubsection*{Step II}
       By naturality of the quantitative Mayer-Vietoris boundary map (see Equation \eqref{equ-boundary-natural} of Section \ref{subsect-mv}),  we have 
       \begin{eqnarray*}
        f_{\H_1\cap\H_2,*}^{\la\eps, l_{\eps}\cdot \r_0 }\circ  \partial_{\H_1,\H_2,A,*}^{\eps,\r_0}([u]_{\eps,\r_0})
        &=&  \partial_{\H_1,\H_2,B,*}^{\eps,\r_0}\circ f_{\H,*}^{\eps,\r_0} ([u]_{\eps,\r_0})\\
       % &=& \partial_{\H_1,\H_2,B,*}^{\eps,\r_0}([\widetilde{f_\H}(u)]_{\eps,\r_0})\\
                                  &=& \partial_{\H_1,\H_2,B,*}^{\eps,\r_0}([\widetilde{f_\H}(u)]_{\eps,\r_0})\\
            &=& 0.
       \end{eqnarray*}
             In particular, 
         \begin{eqnarray*}   f_{\H_1\cap\H_2,*}\circ \iota_*^{\la\eps,   l_{\eps}\cdot \r_0}\circ    \partial_{\H_1,\H_2,A,*}^{\eps,\r_0}([u]_{\eps,\r_0})&=&
      \iota_*^{\la\eps,   l_{\eps}\cdot \r_0}   \circ  f_{\H_1,\H_2,A,*}^{\la\eps,   l_{\eps}\cdot \r_0} \circ    \partial_{\H_1,\H_2,A,*}^{\eps,\r_0}([u]_{\eps,\r_0} )\\
      &=&0,   
  \end{eqnarray*}
  and since   $f_{\H_1\cap\H_2,*}$ is injective, we deduce from Lemma \ref{lemma-approx}   that there exists a compact $\G$-order $\r$ containing   $l_{\eps}\cdot \r_0$ such that
$$ \iota_*^{-, \la^2\eps,\r}\circ    \partial_{\H_1,\H_2,A,*}^{\eps,\r_0}([u]_{\eps,\r_0})=0.$$ 
\subsubsection*{Step III}
  According to Lemma \ref{lemma-bound1}, up to stabilisation and up to replacing $\r$ by $l_{\la^2\eps} \cdot \r$, there exists
  two   $\la^3\eps$-$\r$-unitaries $w_1$ and $w_2$ in $\widetilde{A\rtr\H}$ such that   
  \begin{itemize}
  \item $w_i-1$ is in $A\rtr\H_i$ for $i=1,2$;
  \item $\|u-w_1w_2\|< \la^3\eps$.
  \end{itemize}
  In particular, according to the first point of Lemma \ref{cor-example-homotopy}, we have
  \begin{equation}\label{equation-u}
[u]_{3 \la^3\eps,2\cdot\r}=\jmath_{1,A,*}^{3\la^3\eps,2\cdot\r}([w_1]_{3 \la^3\eps,2\cdot\r})+\jmath_{2,A,*}^{3 \la^3\eps,2\cdot\r}([w_2]_{3 \la^3\eps,2\cdot\r})\end{equation}
in $K_1^{3 \la^3\eps,2\cdot\r}(A\rtr\H)$.  
Moreover, we have
\begin{equation*}
\|v_1v_2-\widetilde{f_\H}(w_1)\widetilde{f_\H}(w_2)\|< 2 \la^3\eps  
\end{equation*}
and in consequence,
\begin{equation*}
\|v_1^*\widetilde{f_\H}(w_1)-v_2\widetilde{f_\H}(w^*_2)\|< 8\la^3\eps.  
\end{equation*}
The CIA-condition with coercivity $c=2$ implies that up to replace $\r$ by $2\cdot\r$, there exists an element $v$ in $1+B\rtr\r_{/\H_1\cap\H_2}$  such that
\begin{equation*}\label{equation-CIA-1}
\|v-v_1^*\widetilde{f_\H}(w_1)\|< 16\la^3 \eps.   \end{equation*}
 and
 \begin{equation*}\label{equation-CIA-2}
\|v-v_2\widetilde{f_\H}^*(w_2)\|< 16 \la^3\eps.   \end{equation*}   
In particular,
$v$ is a $\la'\eps$-unitary with $\la'=64\la^3$. Moreover, $v$ is homotopic to $v_1^*\widetilde{f_\H}(w_1)$ as a
$\la'\eps$-$\r$-unitary in $B\rtr\H_1$ and homotopic to $v_2\widetilde{f_\H}^*(w_2)$ as a
$\la'\eps$-$\r$-unitary in $B\rtr\H_2$.
By  surjectivity of $f_{\H_1\cap\H_2,*}$ and in view of Lemma \ref{lemma-approx}, there exists a compact $\G$-order $\r'$ containing $\r$ and an element
$z$ in $K_1^{\la\la'\eps,\r'}(A\rtr(\H_1\cap\H_2))$ such that
$$f_{\H_1\cap\H_2,*}^{\la\la'\eps,\r'}(z)=[v]_{\la\la'\eps,\r'}$$  in $K_1^{\la\la'\eps,\r'}(B\rtr(\H_1\cap\H_2))$.
\subsubsection*{Step IV}
Let us set $$z_1=\jmath_{1,2,A,*}^{\la\la'\eps,\r'}(z)$$ and
$$z_2=\jmath_{2,1,A,*}^{\la\la'\eps,\r'}(z).$$ 
 We deduce from the discussion at the end of  the  previous  step that
\begin{eqnarray*}
f_{\H_1,*}\circ\iota^{\la\la'\eps,\r'}_*(z_1)&=&\iota^{\la\la'\eps,\r'}_*\circ f_{\H_1,*}^{\la\la'\eps,\r'}(z_1)\\
&=& \iota^{\la\la'\eps,\r'}_*([v_1^*\widetilde{f_{\H_1}}(w_1)]_{\la\la'\eps,\r'})\\
&=& \iota^{\la\la'\eps,\r'}_*([\widetilde{f_{\H_1}}(w_1)]_{\la\la'\eps,\r'})\\
&=&f_{\H_1,*}\circ  \iota^{\la\la'\eps,\r'}_*([w_1]_{\la\la'\eps,\r'}),
\end{eqnarray*}where the third equality holds because    $v_1$ is homotopic to $1$ as a  $\la\eps$-$l_{\eps}\cdot \r_0$-unitary in   $\widetilde{B\rtr\H_1}$. 
Since $f_{\H_1,*}$ is one-to-one, we get that 
$$\iota^{\la\la'\eps,\r'}_*(z_1)=\iota^{\la\la'\eps,\r'}_*([w_1]_{\la\la'\eps,\r'})$$ and similarly, 
$$\iota^{\la\la'\eps,\r'}_*(z_2)=-\iota^{\la\la'\eps,\r'}_*([w_2]_{\la\la'\eps,\r'}).$$ 
\subsubsection*{StepV}
According to Lemma \ref{lemma-approx}, there exists a compact $\G$-order $\r''$ containing $\r'$ and such that
$$\iota^{-,\la^2\la'\eps,\r''}_*(z_1)= [w_1]_{\la^2\la'\eps,\r''}$$ and 
$$\iota^{-,\la^2\la'\eps,\r''}_*(z_2)=- [w_2]_{\la^2\la'\eps,\r''}$$
From  Equation \eqref{equation-u}, we deduce
\begin{eqnarray*}
[u]_{\la^2\la'\eps,\r''}&=&
 \jmath^{\la^2\la'\eps,\r''}_{1,A,*}\circ \iota^{-,-}_*(z_1)- \jmath^{\la^2\la'\eps,\r''}_{2,A,*}\circ \iota^{-,-}_*(z_2)\\
&=&  \iota^{-,\la^2\la'\eps,\r''}_*\circ\jmath^{\la^2\la'\eps,\r''}_{1,A,*}(z_1)- \iota^{-,\la^2\la'\eps,\r''}_*\circ\jmath^{\la^2\la'\eps,\r''}_{2,A,*}(z_2)\\
&=&  \iota^{-,\la^2\la'\eps,\r''}_*\circ(\jmath^{\la^2\la'\eps,\r''}_{1,A,*}\circ \jmath_{1,2,A,*}^{\la\la'\eps,\r'}-\jmath^{\la^2\la'\eps,\r''}_{2,A,*}\circ \jmath_{2,1,A,*}^{\la_0\la'\eps,\r'})(z)\\
&=&0
\end{eqnarray*}
and hence
$$x=\iota_*^{\la^2\la'\eps,\r''}[u]_{\la^2\la'\eps,\r''}=0.$$

\medskip

\subsection{Surjectivity part} Let us set $\al_0=9\la^5$. In view of   Lemma \ref{lemma-approx}, let us  prove that for every $\eps$ in
$(0,\frac{1}{4 \al_0})$, any $\G$-order $\r_0$ and any $y$ in $K_1^{\eps,\r_0}(B\rtr G)$ , there exists a compact $\G$-order $\r_1$ containing $\r_0$ and
$x$ an element in $K_1^{\al_0\eps,\r_1}(A\rtr G)$ such that $f_{\H,*}^{\eps,\r_1}(x)=\iota_{*}^{\eps,\al_0\eps,\r_0,\r_1}(y)$ in $K_1^{\al_0\eps,\r_1}(B\rtr G)$.    We divide this proof in four steps.
\subsubsection*{Step I}
Up to replacing $\r_0$ by $21\eps$ and in view the first point of Lemma \ref{lemma-almost-canonical-form-odd}, we can assume that there exists an $\eps$-$\r_0$-unitary $u'$ in  $I_{N'}+M_{N'}({B\rtr\H})$ for some integer $N'$ such that 
 $y=[u']_{\eps,\r_0}$.
 Let $(V_1,V_2,\H_1,\H_2)$ be 
     a  $({6l_{\eps}}\cdot\r_0)$-decomposition for $\H$ with $\H_1$ and $\H_2$ in $\D$.  Since $f_{\H_1\cap\H_2,*}$ is onto and according to Lemma \ref{lemma-approx}, there exists a compact $\G$-order $\r$ containing $l_{\eps}\cdot\r_0$ and an element $x_{1,2}$ in $K_0^{\la^2\eps,\r}(A\rtr(\H_1\cap\H_2))$ such that
     \begin{equation*}
     f_{ \H_1\cap\H_2,*}^{\la^2\eps,\r}(x_{1,2})=\iota_*^{-,\la^2\eps,\r}\circ\partial_{\H_1,\H_2,B,*}^{\eps,\r_0}(y).
     \end{equation*}  
     Let us set $x_1=\jmath_{1,2,A,*}^{\la^2\eps,\r}(x_{1,2})$ in $K_0^{\la^2\eps,\r}(A\rtr\H_1)$ and 
     $x_2=\jmath_{2,1,A,*}^{\la^2\eps,\r}(x_{1,2})$ in $K_0^{\la^2\eps,\r}(A\rtr\H_2)$.    
     Then we have 
     \begin{eqnarray*}
     f_{\H_1,*}\circ \iota_*^{\la^2\eps,\r}(x_1)&=&\iota_*^{\la^2\eps,\r}\circ  f_{\H_1,*}^{\la^2\eps,\r}(x_{1})\\
     &=&\iota_*^{\la^2\eps,\r}\circ  f_{\H_1,*}^{\la^2\eps,\r}\circ \jmath_{1,2,A,*}^{\la^2\eps,\r}(x_{1,2})\\  
     &=&\iota_*^{\la^2\eps,\r}\circ \jmath_{1,2,B,*}^{\la^2\eps,\r}\circ  f_{\H_1\cap\H_2,*}^{\la^2\eps,\r}(x_{1,2})\\       
      &=&\iota_*^{\la^2\eps,\r}\circ \jmath_{1,2,B,*}^{\la^2\eps,\r}\circ 
        \iota_*^{-,-}\circ\partial_{\H_1,\H_2,B,*}^{\eps,\r_0}(y)\\
        &=&\iota_*^{\la^2\eps,\r}\circ 
        \iota_*^{-,-}\circ \jmath_{1,2,B,*}^{\la_\eps,l_{\eps}\r_0}\circ\partial_{\H_1,\H_2,B,*}^{\eps,\r_0}(y)\\
        &=&0     
           \end{eqnarray*} and in the same way  $f_{\H_2,*}\circ \iota_*^{\la^2\eps,\r}(x_2)= 0$.  
           Since $f_{\H_1,*}$ and $f_{\H_2,*}$ are one-to-one, we deduce that 
         $\iota_*^{\la^2\eps,\r}(x_1)=0$     and    $\iota_*^{\la^2\eps,\r}(x_2)=0$. According to Lemma \ref{lemma-approx}, there exists a compact $\G$-order $\r'$ containing $\r$ and such that
         \begin{equation*}
         \jmath_{1,2,A,* }^{\la^3\eps,\r'}\circ\iota_*^{\la^2\eps,\r,-}(x_{1,2})=0
         \end{equation*}
         and 
         \begin{equation*}
         \jmath_{2,1,A,* }^{\la^3\eps,\r'}\circ\iota_*^{\la^2\eps,\r,-}(x_{1,2})=0.
         \end{equation*}     
         \subsubsection*{Step II}
         In view of Lemma    \ref{lemma-almost-canonical-form}, there exists for some positive integers $n$ and $N$ with $n\leqslant N$  a $9\la^2\eps$-$\r$-projection  in $\diag(I_n,0)+M_N({A\rtr\H})$ such that 
         \begin{equation*}
         \iota_*^{\la^2\eps,9\la^2\eps,\r}(x_{1,2})=[p,n]_{9\la^2\eps,\r}
         \end{equation*}       
         in $K_0^{9\la^2\eps,\r}(A\rtr(\H_1\cap\H_2))$.            According to Lemma 
       \ref{lemma-exactness-inclusion} and up to stabilisation, there exists four $9\la^3\eps$-$l_{9\la^2\eps}\r$-unitaries, $v_1$ and $v_2$ in
       $I_N+M_N({A\rtr\H})$, $u_1$ in $M_n(\widetilde{A\rtr\H})$ and $u_2$ in $M_{N-n}(\widetilde{A\rtr\H})$ such that
       \begin{itemize}
       \item $\|v^*_1\diag(I_n,0) v_1-p\|<9\la^3\eps$;
       \item  $\|v_2\diag(I_n,0) v^*_2-p\|<9\la^3\eps$;   
       \item $\|\diag(u_1,u_2)-v_1v_2\|<9\la^3\eps$;
       \item for $i=1,2$, then $v_i$ is connected to $I_N$ by a homotopy of $9\la^3\eps$-$l_{9\la^2\eps}\r$-unitaries in $I_N+M_N(A\rtr\H_i)$.
       \end{itemize}   
       
       \subsubsection*{Step III}
     Let    $u'$ be an $\eps$-$\r_0$-unitary   in  $I_{N'}+M_{N'}(\widetilde{B\rtr\H})$ for some integer $N'$ such that 
 $y=[u']_{\eps,\r_0}$.
By construction of the controlled Mayer-Vietoris boundary applied to $-y=[u^*]_{\eps,\r_0}$, there exists two 
  $\la\eps$-$l_{\eps}$-unitaries $w_1$ and $w_2$ in $M_{2N'}(\widetilde{B\rtr\H})$ and $q$ an  
  $\la\eps$-$l_{\eps}$-projection in      $\diag(I_{N'},0)+M_{2N'}({B\rtr\H})$ such that     
  \begin{itemize}
  \item $w_i-I_{2N'}$ lies in $M_{2N'}(B\rtr\H_i)$ for $i=1,2$;
  \item $\|diag({u'}^*,u')-w_1w_2\|<\la\eps$;
  \item $\|w_1^*\diag(I_{N'},0)w_1-q\|<\la\eps$ and 
   $\|w_2\diag(I_{N'},0)w^*_2-q\|<\la\eps$; 
   \item $-\partial^{\eps,\r_0}_{\H_1,\H_2,B,*}(z)=[q,N']_{\la\eps,l_{\eps}\cdot\r_0}$  .
   \end{itemize}    
   \subsubsection*{Step IV}
   If we apply Lemma \ref{lemma-bound1} to 
   $\diag(\widetilde{f_\H}(u_1),{u'}^*)$, $\diag(\widetilde{f_\H}(u_2),{u'})$     and to the matrices obtained from
   $\diag(\widetilde{f_\H}(v_1),w_1)$, $\diag(\widetilde{f_\H}(v_2),{w_2})$    and $\diag(\widetilde{f_\H}(p),q)$ by flipping the coordinates $n+1,\ldots,N$ and $N+1,\ldots,N+N'$, we see that   up to replacing $\r'$ by 
   $l_{9\la^3\eps}\r'$, there exist for some integer and for $i=1,2$ a $9\la^4\eps$-$\r'$-unitary $v'_i$ in $I_{N''}+M_{N''}(B\rtr\H_i)$   such that
   \begin{equation*}
   [v'_1]_{9\la^4\eps,\r'}+[v'_2]_{9\la^4\eps,\r'}=[\widetilde{f_\H}(u_1)]_{9\la^4\eps,\r'}-\iota_*^{\eps,9\la^4\eps,\r_0,\r'}(y).
   \end{equation*}
   Since $f_{\H_1,*}$ and $f_{\H_2,*}$ are onto and according to Lemma \ref{lemma-approx}, there exist a compact $\G$-order $\r_1$ containing $\r'$ and for $i=1,2$ an element $x_i$ in $K_1^{{9\la^5\eps,\r_1}}(A\rtr\H_i)$ such that
   $f_{\H_i,*}^{9\la^5\eps,\r_1}(x_i)=-[v'_i]_{9\la^5\eps,\r_1}$ in  $K_1^{{9\la^5\eps,\r_1}}(B\rtr\H_i)$. Then 
   we have
   \begin{equation*}
   \iota_*^{\eps,9\la^5\eps,\r_0,\r_1}(y)=f_{\H,*}^{5\la^5\eps,\r_1} (\jmath_{1,A,*}^{5\la^5\eps,\r_1}(x_1)+\jmath_{2,A,*}^{5\la^5\eps,\r_1} (x_2)+ 
   [u_1]_{5\la^5\eps,\r_1})    \end{equation*}
   and hence $f_{\H,*}$ is onto.
            \end{proof}
            
     \subsection{Extension to Kasparov product}
     The aim of this section is to extend Theorem \ref{theorem-main}
      to morphisms induced in $K$-theory by right  Kasparov product. 
      Indeed, this a consequence of the following standard fact  which says  that up to $KK$-equivalence, a Kasparov element  is equivalent to $C^*$-algebra homomorphism (see \cite{mn} for an approach via triangulated categories).
         \begin{lemma}\label{lemma-eq-morphism}
   Let $\G$  be a  locally   groupoid  provided with a Haar system,  let $A$ and $B$ be $\G$-algebras and let $z$ be an element in
   $KK^\G_*(A,B)$. Then there exists 
   \begin{itemize}
   \item $A'$ and $B'$ two $\G$-algebras;
   \item $f:A'\to B'$  a homomorphism of $\G$-algebras;
   \item $\al$ in $KK^\G_*(A,A')$ and $\beta$ in $KK^\G_*(B',B)$ invertible elements,
   \end{itemize}
   such that $$z=f_*(\al)\ts_{B'} \beta.$$
   \end{lemma}
   \begin{proof}
   Let us first prove the result for $z$ in $KK_1^\G(A,B)$. The imprimitivity $\Kp(\L^2(\G,A))$-$A$-bimodule $\L^2(\G,A)$ gives rise  an invertible element $[\mathcal{M}]$ in $KK_*^\G(\Kp(\L^2(\G,A)),A)$  and hence, this reduces to prove the result for $[\mathcal{M}]\ts_A z$. According to Lemma 3.5 of the appendix of \cite{laff-imj} (see also \cite[Section 5]{bonicke}), this amount to prove the result for any element $z$ in $KK_1^\G(A,B)$ that can be represented by a Kasparov $K$-cycle $(\E,\pi,F)$ such that  $F:\E\to\E$   is  $\G$-equivariant. Up to adding a degenerated Kasparov $K$-cycle, we can assume without loss of generality that 
   the linear space generated by $\{\langle \xi,\nu\rangle,\, \xi\text{ and }\nu\text{ in }\E\}$ is dense in $B$.  Let us set $P=\frac{1}{2}(F+Id_\E)$ and let us consider the $\G$-algebra
 \begin{equation*} E_P=\{(a,f)\text{ in } A\oplus \mathcal{L}(\E)\text{ such that } T-P\cdot\pi(a)\cdot P\text{ belongs to }\Kp(\E)\}.\end{equation*}
 Then the projection on the first factor of $E_P$ gives rise to an extension of $\G$-algebra 
 \begin{equation}\label{eq-semi-split}0\lto \Kp(\E)\lto E_P\lto A\lto 0$$ semi-splited by 
 $$A\lto E_P;a\mapsto ( a,P\cdot\pi(a)\cdot P).\end{equation} Let $[\M]$ be the element in $KK_1^\G(\Kp(\E),B)$  corresponding to the
 $\Kp(\E)$-$B$-imprimitivity bimodule $\E$. Then $[\M]$ is invertible and $z\ts_B[\M]^{-1}$ is the class in $KK_1^\G(A,\Kp(\E))$ of the semi-split extension \eqref{eq-semi-split} . Hence this amounts to prove that the lemma holds for the class
 $[\partial_{I,A}]$ in $KK_1^\G(A/I,I)$ of any semi-split extension $0\lto I\lto A\lto A/I\lto 0$. We proceed by using the mapping cone. For $B$ a $\G$-algebra let us set $$B(0,1]=\{f:[0,1]\to \C\text{ continuous such that }f(0)=0\},$$
 $$B(0,1)=\{f:[0,1]\to \C\text{ continuous such that }f(0)=f(1)=0\},$$   and let us consider the class $[\partial_B]$ in $KK_1^\G(B, B(0,1))$ of the semi-split extension of $\G$-algebras
 $$0\lto B(0,1)\lto  B(0,1]\stackrel{ev_1}\lto B \lto 0,$$ where $ev_1:B(0,1]\to B$ is the evaluation at $1$. Recall that
 $[\partial_B]$ is invertible. For a semi-split extension
 $$0\lto I\lto A\lto A/I\stackrel{q}{\lto} 0,$$  we define the mapping cone algebra of $A$ by 
 $$C_q=\{(x,f)\in A\oplus A/I(0,1]\text{ such that } f(1)=q(x)\}.$$
 Let us consider the morphisms of $\G$-algebras
 $$e_q:I\lto C_q; x\mapsto (x,0).$$ and
 $$\phi_q:A/I(0,1)\lto C_q; f\mapsto (0,f).$$
 According to \cite{sk-exact-sequences}, the element $[e_q]$ in $KK_*^\G(I,C_q)$ induced by $e_q$  is invertible and moreover,
 $$e_{q,*}[\partial_{I,A}]=\phi_{q,*}[\partial_{A/I}].$$ Hence we have 
 
 $$[\partial_{I,A}] =\phi_{q,*}[\partial_{A/I}]\ts_{C_q}  [e_q]^{-1}.$$ Since $[\partial_{A/I}]$ is an invertible element in
 $KK_1^\G(A/I,A/I(0,1))$, we deduce that the conclusion of the lemma holds for $[\partial_{I,A}]$.
 
 \smallbreak
 
 Let us prove now that the result holds in the even case. Let $z$ be an element in $KK_1^\G(A,B)$. Noticing that $[\partial_A]$ is invertible in $KK^\G_1(A,A(0,1)$ and applying the odd case to
 $[\partial_A]^{-1}\otimes_A z$, we deduce the result for $z$.
 \end{proof}
 
 As a consequence, we can extend  Theorem \ref{theorem-main} to $KK$-elements.
 Recall that for any locally   groupoid  $\G$  provided with a Haar system, then
 $$J_\G:KK^\G_*(\bullet,\bullet)\to KK_*(\bullet\rtr\G,\bullet\rtr\G)$$ stands for the Kasparov transformation. For any $\G$-algebras $A$ and $B$ and for 	any element  $z$   in $KK^\G_*(A,B)$, we denote by $$\ts J_\G(z):K_*(A\rtr\G)\to K_*(B\rtr\G)$$ the morphism induced by Kasparov right multiplication by $J_\G(z)$.
 
    \begin{corollary}\label{cor-extension}
     Let $\G$  be a  locally   groupoid  provided with a Haar system, let $A$ and $B$ be $\G$-algebras and let $z$ be an element in
     $KK^\G_*(A,B)$.    Let us assume that there exists a  subset $\D$ of relatively clopen groupoids of $\G$, closed under taking relatively  clopen subgroupoids and  such that
  \begin{enumerate}
   \item $\G$ has finite $\D$-complexity.
   \item for any subgroupoid $\H$, the morphism  $$\ts J_\H(z_{/\H}):K_*(A\rtr\H)\to K_*(B\rtr\H)$$ is an isomorphism.
  \end{enumerate}
  Then  $$\ts J_\G(z):K_*(A\rtr\G)\to K_*(B\rtr\G)$$  is an isomorphism.
  \end{corollary}
  \begin{proof}Let $A'$ and $B'$ be $\G$-algebras, let $f:A'\to B'$ be a morphism of $\G$-algebras and let $\al$ in $KK^\G_*(A,A')$ and $\beta$ in $KK^\G_*(B',B)$  be invertible elements as in Lemma \ref{lemma-eq-morphism}. Then $$\ts J_\G(z):K_*(A\rtr\G)\to K_*(B\rtr\G)$$ is an isomorphism if and only if $$f_{\G,*}:K_*(A'\rtr\G)\to K_*(B'\rtr\G)$$ is an isomorphism and in the same way 
  $$\ts J_\H(z_{/\H}):K_*(A\rtr\H)\to K_*(B\rtr\H)$$ is an isomorphism if and only if
  $$f_{\H,*}:K_*(A'\rtr\H)\to K_*(B'\rtr\H).$$ The corollary is then consequence of Theorem
  \ref{theorem-main} applied to $f:A'\to B'$ for any $\H$ in $\D$.\end{proof}
  \section{Application the Baum-Connes conjecture}
  In this section, we show that for groupoids which admit a $\ga$-element in sense of 
  \cite{tuhyp},  then the Baum-Connes conjecture is closed under closed decomposition.

   \subsection{The Baum-Connes conjecture for groupoids}
  Let us recall first the statement of the Baum-Connes conjecture for  groupoids. Let $\G$ be a locally compact groupoid provided with the Haar system, let $A$ be a $\G$-algebra and let $A\rtr\G$ be the reduced crossed product of $A$ by $\G$ (with respect to the given Haar system). Then the Baum-Connes conjecture for the pair $(A,\G)$ is the  claim that the assembly map
  $$\mu_{A,\G}:K_*^{top}(\G,A)\lto K_*(A\rtr\G)$$ is an isomorphism,  the left-hand side being  the topological $K$-theory for the groupoid $\G$ with coefficient, defined as the inductive limit
  $$\lim_X KK_*^\G(C_0(X),A),$$ where $X$ runs through $\G$-compact subsets of the universal example for proper actions of $\G$ (see \cite{tuhyp}  for  a complete description of the Baum-Connes conjecture in the setting of groupoids). Although the conjecture holds for a large class of pair $(A,\G)$, (e.g if $\G$ is an amenable groupoid), counterexamples  have been exhibited by N. Higson, V. Lafforgues and G. Skandalis in \cite{hls}. 
%   The following result was proved in \cite{bonicke}.\red{v\'erifier}
%  \begin{theorem}\label{thm-BC-subgroupoid}
%   Let $\G$ be   a  locally   groupoid  provided with a Haar system. If $\G$  satisfy the Baum-Connes conjecture with coefficients, then any relatively clopen subgroupoid of $\G$ satisfies the Baum-Connes conjecture with coefficients.
%  \end{theorem}
   \subsection{The case of groupoids admitting a $\ga$-elements}  
   The concept of $\ga$-element was introduced by G. Kasparov in \cite{kas}  in order to prove  the Novikov conjecture for discrete subgroups of almost  connected groups. He showed that for an  almost connected group $G$ acting on a $C^*$-algebra $A$ strongly continuously  by automorphisms, then the image of the Baum-Connes assembly map is the range of $\ga$ acting on $K_*(A\rtr G)$ as an idempotent. The notion of $\ga$-element was extended to groupoid by J. L. Tu in \cite{tuhyp}, where he developed an abstract  setting for such an element.

 \begin{definition}\label{definition-gamma-element}
 A locally compact groupoid admits a $\ga$-element if there exists an element $\ga$ in $KK_*^\G(C_0(X),C_0(X))$,  a  proper  $\G$-space $Z$,    a $Z\rt\G$-algebra $A$, an element  $\eta$ in
 $KK_*^\G(C_0(X),A)$ and an element $D$ in $KK_*^\G(C_0(X),A)$ such that  
 \begin{itemize}
 \item $\ga=\eta\ts_A D$;
 \item $p^*_{Z'}\ga=1$ in $K_*^{\G\rt Z'}(C_0(Z'),C_0(Z'))$ for every proper  $\G$-space $Z'$, where $p_{Z'}:Z'\rtr\G\to   \G$ is the forgetful map.
 \end{itemize} \end{definition}
 Such an element if it exists is unique and is called a $\ga$-element. As in the case of the $\gamma$-element of Kasparov, a $\gamma$-element   is an idempotent of  $KK_*^\G(C_0(X),C_0(X))$ and as such  acts as an idempotent on $K_*(A\rtr\G)$.  This idempotent is given by right Kasparov product by $J_\G(\tau_A(\ga))$, where 
 $\tau_A(\ga)\in KK^\G_*(A,A)$ is obtained  by tensorization over $C(X)$ by $A$.  Moreover, it is related to the Baum-Connes conjecture in the following way:  

 \begin{proposition}
 Let $\G$ be a locally compact groupoid provided with a Haar system admitting a $\gamma$-element and  let $A$ be a $\G$-algebra. Then the following assertions are equivalent : 
 \begin{enumerate}
 \item $\mu_{A,\G}:K_*^{top}(\G,A)\lto K_*(A\rtr\G)$ is an isomorphism ;
 \item  $J_\G(\tau_A(\ga))$ acts as the identity by right Kasparov product on $K_*(A\rtr\G)$.
 \end{enumerate}
 \end{proposition}
 \begin{remark}
 Since $J_\G(\tau_A(\ga))$ is an idempotent, it acts as the identity by right Kasparov product on $K_*(A\rtr\G)$ if and only if it acts as an isomorphism.
 \end{remark}
% If $\G$ is a groupoid equipped with a Haar system ç, and if $\H$ is an open subgroupoid of $\G$, then 
The restriction of a $\ga$-element to a relatively clopen subgroupoid is a $\ga$-element.
\begin{lemma}\label{lemma-gamma-restriction}
Let $\G$ be a locally compact groupoid and let  $\H$ be  a relatively clopen subgroupoid of $\G$. If $\G$ admits a $\gamma$-element, then the restriction of $\gamma$ to $\H$ is a   
 $\gamma$-element for $\H$.
 \end{lemma}
    \begin{proof}Let us denote respectively by $X$ and $Y$ the  space of units  of $\G$ and $\H$.
    Let $Z$ be a proper $\G$-space, let $A$ be a $\G\rt Z$-algebra, let  $\eta$ be an element in $KK_*^\G(C_0(X),A)$ and let $D$ be an element in $KK_*^\G(A,C_0(X))$  as in
    Definition \ref{definition-gamma-element}. According to  Remark \ref{remark-relatively-clopen-action} and to Corollary \ref{corollary-subgroupoid-proper}, the proper action of $\G$ on $Z$ restricts to a proper action of 
        $\H$ on $Z_Y$. Let $A_{/Z_Y}$ be the restriction of $A$ to $Z_Y$. According to the second point of Example \ref{example-action-algebra}, then 
     $A_{/Z_Y}$ is a $\H\rt Z_Y$ algebra. Let $\gamma_{/\H}$  in $KK_*^\H(C_0(Y),C_0(Y))$,  $\eta_{/\H}$  in $KK_*^\H(C_0(Y), A_{/Z_Y})$ and 
 $\D_{/\H}$  in  $KK_*^\H(A_{/Z_Y},C_0(Y))$  be respectively the restriction of $\gamma$, $\eta$ and  $D$ to $\H$ (i.e induced by functoriality in the groupoids by the inclusion  $\H\hookrightarrow\G$). Since the restriction respects  Kasparov products, we deduce that $\ga_{/\H}=\eta_{/\H}\ts_{A_{/Z_Y}} D_{/\H}$. Let us check the second point of the definition of a $\gamma$-element. Let $Z'$ be a proper $\H$-space, let $Z''=\G\times_\H Z'$ be the proper induced $\G$-space (see Section \ref{subsec-induced-actions}) and let us recall that  $p_{Z''}:\G\lt Z'' \to \G$ stands for the forgetful map. We have by definition of a $\gamma$-element that $p^*_{Z''}\ga=1$ in $K_*^{Z''\rt\G}(C_0(Z''),C_0(Z''))$. We have an obvious  inclusion of groupoids
 $$\H\rt Z'\hookrightarrow \G\rt Z'';\,(\gamma,z)\mapsto (\ga,[u_{p_{Z'}}(z),z])$$ which pulls back  $p^*_{Z''}\ga$ to $p^*_{Z'}\ga_{/\H}$ and hence 
  $p^*_{Z'}\ga_{/\H}=1$ in $K_*^{\H\rt Z'}(C_0(Z'),C_0(Z'))$. We conclude that  $\ga_{/\H}$ is a 
  $\gamma$-element for $\H$.
    \end{proof}
    
    \begin{remark}\label{rem-bcc-subgroupoid}
    As a consequence and using induced algebras \cite{bonicke}, we can prove  that if $\G$ is  a locally compact groupoid which admits a $\gamma$-element and which satisfies the Baum-Connes conjecture with coefficients, then any relatively clopen subgroupoid of $\G$ satisfies 
     the Baum-Connes conjecture with coefficients.
     \end{remark}
    An action groupoid of a groupoid with a $\ga$-element has a $\ga$-element.
    \begin{lemma}
    Let $\G$ be a locally compact groupoid and let  $Y$ be  a locally compact (left) $\G$-space. If $\G$ admits a $\ga$-element, then  the action groupoid 
     $\G\lt Y$ admits a $\ga$-element.
   \end{lemma}
   \begin{proof}
   Let us denote by $X$   the space of units   of $\G$ and let $q_Y:Y\to X$ be the anchor map for 
   the $\G$-action on $Y$.
    Let $Z$ be a proper $\G$-space, let $A$ be a $\G\lt Z$-algebra, let  $\eta$ be an element in $KK_*^\G(C_0(X),A)$ and let $D$ be an element in $KK_*^\G(A,C_0(X))$  as in
    Definition \ref{definition-gamma-element}. Let $p_Y:\G\rt Y\to\G$ be the forgetful map with respect to the $\G$-action on $Y$.  According to the fourth  point of Remark \ref{remark-cross-product-action}, we see that  $Z\times_X Y$ is a proper $\G\rt Y$-space with anchor map 
    $q_{Z\times_X Y}:Z\times_X Y\to Y$ given by the projection on the second factor.
    Consider then the elements $\ga_Y=p^*_Y\ga$ in $KK_*^{\G\lt Y}(C_0(Y),C_0(Y))$, 
    $\eta_Y=p^*_Y\eta$ in $KK_*^{\G\lt Y}(C_0(Y),q_Y^*A)$ and $D_Y=p^*_YD$ in
    $KK_*^{\G\lt Y}(q_Y^*A,C_0(Y))$. Using  the second point of Example \ref{example-action-algebra}, we see that  $q_Y^*A=A\otimes_{C_0(X)} C_0(Y)$ is a 
    $({\G\lt Y})\lt ( Z\times_X Y)$-algebra and since $p^*_Y$ preserves Kasparov products, 
    we have
    $$\ga_Y=\eta_Y\ts_ {q_Y^*A}D_Y.$$
    Let us check now the second condition of  Definition \ref{definition-gamma-element}.
    Let $Z'$ be a proper $\G\lt Y$-space. According to the first and the third point of Remark
    \ref{remark-cross-product-action}, we see  that $Z'$ is a proper $\G$-space equipped
    with a $\G$-map $Z'\to Y$. Let $$p_{Z'}:(\G\lt Y)\lt Z'\to \G\lt Y$$ be the forgetful map with respect to the $\G\lt Y$-action on $Z'$. Then we have 
    \begin{eqnarray*}
    p_{Z'}^*\ga_Y&=&p_{Z'}^*(p_Y^*\ga)\\
    &=& (p_Y\circ p_{Z'})^*\ga
    \end{eqnarray*}
    But under the identification between  $(\G\lt Y)\lt Z'$ and $\G\lt Z'$ of the second point of
    Remark
    \ref{remark-cross-product-action},
    then $$p_Y\circ p_{Z'}:(\G\lt Y)\lt Z'\to \G$$ corresponds to the forgetful map
    $\G\lt Z'\to\G$ with respect to the  proper $\G$-action on $Z'$. From this we deduce that
    $p^*_{Z'}\ga_Y=1$ in $KK_*^{(\G\lt Y)\lt  Z'}(C_0(Z'),C_0(Z'))$ and hence $\ga_Y$ is a $\ga$-element for $\G\lt Y$.

    \end{proof}
    \begin{example}
   The following examples of locally compact group $G$ are known to have a $\ga$-element
    \begin{enumerate}
    \item if $G$ acts properly  on a simply connected manifold with non-positive sectional curvature \cite{kas};
    \item if $G$ is (a closed subgroup of) an almost connected group \cite{kas};
    \item if $G$ is groups acting properly on an Euclidean buildings \cite{ks}.
    \end{enumerate}
    For any action of such a group $G$ on a locally compact space $X$, then the action groupoid $G\lt X$ has a $\ga$-element.
    \end{example}

  \subsection{Baum-Connes conjecture and coarse decomposability}

  \begin{theorem}\label{thm-BC-coarse-decomp}
  Let $\G$  be a  locally   groupoid  provided with a Haar system   which moreover admits a  $\ga$-element in sense of 
  \cite{tuhyp} and let $A$ be a $\G$-algebra. Assume that there exists a  subset $\D$ of relatively clopen subgroupoids of $\G$,  closed under taking relatively  clopen subgroupoids
  such that
  \begin{enumerate}
  \item every groupoid in $\D$ satisfies the Baum-Connes conjecture with coefficients in $A$;
  \item $\G$ has finite $\D$-complexity.
  \end{enumerate}
  Then $\G$ satisfies the Baum-Connes conjecture with coefficients in $A$.
  \end{theorem}
  \begin{proof}
  This is a consequence of Corollary \ref{cor-extension} applied to $\tau_A(\ga)$ and of
  Lemma \ref{lemma-gamma-restriction}, by noticing that $\tau_A(\ga)_{/\H}=\tau_A(\ga_{/\H})$ for any relatively clopen subgroupoids $\H$ in $\D$.
  \end{proof}
  We end this paper with an application to the Baum-Connes conjecture with coefficients.

   \begin{corollary}
  Let $\G$  be a  locally   groupoid  provided with a Haar system and which moreover admits a $\ga$-element in sense of 
  \cite{tuhyp}. Assume that there exists a  subset $\D$ of relatively clopen subgroupoids of $\G$  such that
  \begin{enumerate}
  \item every groupoid in $\D$ satisfies the Baum-Connes conjecture with coefficients;
  \item $\G$ has finite $\D$-complexity.
  \end{enumerate}
  Then $\G$ satisfies the Baum-Connes conjecture with coefficients.
  \end{corollary}
  \begin{proof}
  Let $\D'$ be the set of all relatively clopen subgroupoids of elements of $\D$. According to Remark \ref{rem-bcc-subgroupoid}, any groupoid $\H$ in $\D'$ satisfies the Baum-Connes conjecture with coefficients. Since $\G$ has finite  $\D$-complexity, it has  $\D'$-complexity.
  The result is then a consequence of
  Theorem \ref{thm-BC-coarse-decomp}.
  \end{proof}

 \bibliographystyle{plain}

\end{document}